\documentclass[review]{elsarticle}

\usepackage{latexsym}
\usepackage{graphics}
\usepackage[margin=1.25in]{geometry}
\usepackage{amsmath}
\usepackage{amsfonts}
\usepackage{setspace}
\usepackage{hyperref}

\makeatletter
\def\@opargbegintheorem#1#2#3{\trivlist
   \item[]{\bfseries #1\ #2\ (#3)} \itshape}
\makeatother

\journal{Journal of \LaTeX\ Templates}

\usepackage{numcompress}

\newtheorem{theorem}{Theorem}
\newtheorem{prop}{Proposition}
\newtheorem{conj1}{Conjecture}
\newproof{proof}{Proof}
\newtheorem{lemma}{Lemma}

\numberwithin{equation}{section}
\numberwithin{theorem}{section}
\numberwithin{prop}{section}

\begin{document}

\begin{spacing}{0.89}

\begin{frontmatter}

\title{A generalization of Picard-Lindelof theorem/ the method of characteristics to systems of PDE}

\author[mymainaddress,mysecondaryaddress]{Erfan Shalchian\fnref{myfootnote}}
\address[mysecondaryaddress]{Department of Physics, University of Toronto, 60 St. George St., Toronto ON M5S 1A7, Canada }
\address[mymainaddress]{Department of Mathematics and Department of Physics, Sharif Institute of Technology, Azadi St., Tehran, Iran }
\ead{erfanshalchian@gmail.com,eshalchian@physics.utoronto.ca}

\fntext[myfootnote]{Full name: M. Erfan Shalchian T.}

\begin{abstract}
We generalize Picard-Lindelof theorem/ the method of characteristics to the following system of PDE: $C_{il}(x,y) {\partial y_i / \partial x_l} + {\partial y_i / \partial x_m} = D_i(x,y)$. With a Lipschitz or $C^r$ $C_{il},D_i: [-a, a]^{m} \times [-b, b]^{n} \rightarrow \mathbb{R}$ and initial condition $I_i: [-\bar{a}, \bar{a}]^{m-1} \rightarrow (-b,b)$, $\bar{a} \leq a$, we obtain a local unique Lipschitz or $C^r$ solution $f$, respectively that satisfies the initial condition, $f_i (v, 0 ) = I_i(v)$, $v \in [-\bar{a}, \bar{a}]^{m-1}$. To construct the solution we set bounds on the value of the solution by discretizing the domain of the solution along the direction perpendicular to the initial condition hyperplane. As the number of discretization hyperplanes is taken to infinity the upper and lower bounds of the solution approach each other, hence this gives a unique function for the solution ($Ufs$). A locality condition is derived based on the constants of the problem. The dependence of $C_{il}$, $D_i$ and $I_i$ on parameters, the generalization to nonlinear systems of PDE and the application to hyperbolic quasilinear systems of first order PDE in two independent variables is discussed.
\end{abstract}

\begin{keyword}
\texttt{hyperbolic quasilinear systems of PDE \sep method of characteristics \sep Lipschitz continuity \sep upper and lower bounds }
\end{keyword}

\end{frontmatter}

\tableofcontents
\end{spacing}

\onehalfspacing
\section{Introduction and outline}
The method of characteristics for solving a first order partial differential equation in an unknown function has been known to mathematicians in the past centuries, however, the generalization of this method to systems of first order PDE has remained unknown (e.g.\cite{01}: Chapter VI, Section 7 it is stated that there is no analog of the method of characteristics for systems of first order PDE). In this work we will prove theorems, in particular Theorem \ref{eq:1.1} below, that will generalize the result obtained using the method of characteristics, typically applicable to one equation with one unknown function, to systems of first order PDE which the partial derivatives of each function appear in separate equations. Theorem \ref{eq:1.1} can also be considered as the generalization of the Picard-Lindelof theorem of ODE to PDE. The main result of this work proven is Section \ref{sec:3} is the following Theorem:
\begin{theorem}[A generalization of Picard-Lindelof theorem/ the method of characteristics to systems of PDE]
\label{theorem1.1}
\onehalfspacing
Let $C_{il}, D_i : P \rightarrow \mathbb{R} , \ i = 1, ..., n , l =1, ..., m-1$, $m \geq 2$, be Lipschitz continuous or $C^r$ $( r\geq 1)$ functions defined on the parallelpiped $P \equiv P_1 \times P_2$ with $P_1 \equiv \{ x \in \mathbb{R}^m | \left\Vert x - x_0 \right\Vert_\infty \leq a, x_0 \in \mathbb{R}^m \}$ and $ P_2 \equiv \{ y \in \mathbb{R}^n |  \left\Vert y - y_0 \right\Vert_\infty \leq b, y_0 \in \mathbb{R}^n \}$. And let the Lipschitz continuous or $C^r$ initial condition function $I : V \rightarrow P_2$ for $V \equiv \{ x \in P_1 \vert x_m = x_{0m} , |x_l - x_{0l}| \leq \bar{a} \}$, $0 < \bar{a} \leq a$ and $M_{\left\Vert I-y_0 \right\Vert} < b$ with $M_{\left\Vert I-y_0 \right\Vert} \equiv \max \{\left\Vert I(u) - y_0 \right\Vert_\infty \vert u \in V \}$ be given. The following system of partial differential equations
\begin{flalign}
\label{eq:1.1}
C_{i1}(x,y) {\partial y_i \over \partial x_1} + ... + C_{im-1}(x,y) {\partial y_i \over \partial x_{m-1}} + {\partial y_i \over \partial x_m} = D_i(x,y)
\end{flalign}
has a unique Lipschitz continuous \footnote[1]{By Lipschitz continuous solution we mean a Lipschitz continuous function that solves the system of PDE \ref{eq:1.1} at its differentiable points. By Rademacher theorem (for a proof refer to \cite{04}) a Lipschitz continuous function is differentiable almost everywhere.} or $C^r$ solution respectively, $f : B \rightarrow P_2$ for $V \subset B \subseteq P_1$, $B$ containing a neighbourhood of $V_{int}$, with $V_{int} \equiv \{ x \in P_1 \vert x_m = x_{0m} , |x_l - x_{0l}| < \bar{a} \}$ and $f$ reducing to the initial condition function $I$ on $V$, $f(u) = I(u)$ for $u \in V $.
\end{theorem} 

The proof of Theorem \ref{theorem1.1} is far from trivial. The main difficulty in generalizing the method of characteristics to the system of PDE of the type \ref{eq:1.1} is that the characteristic curves for each equation are distinct therefore it cannot be reduced to systems of ODE. One way to gain control over these characteristics is to set bounds on the value of the solution satisfying an initial condition and the characteristic curves which are distinct for each equation by discretizing the hyperplanes along the direction perpendicular to the initial condition hyperplane. If the bounds are set in an appropriate and optimal way it can be shown that in the limit that the number of discretization hyperplanes is taken to infinity the bounds for the value of the solution and the characteristic curves approach each other, hence this gives a unique function for the solution ($Ufs$).

It should be noted that there is a more general and abstract theorem in hyperbolic systems of partial differential equations that is related to the system of PDE of relation \ref{eq:1.1}, however the conditions of that theorem, being a more general result are not as minimal as the conditions of Theorem \ref{eq:1.1}. For example the differentiability assumptions of that theorem have to increase proportional to the number of independent variables used in the hyperbolic system of PDE in order for the solution to be a bounded ordinary function possessing finite derivatives to a certain order (for more details refer to \cite{03}, Chapter VI, Section 10). On the other hand the conditions of Theorem \ref{theorem1.1} are as minimal as they can be. Another interesting feature of Theorem \ref{theorem1.1} is the method which it is proven with, which is an elegant generalization of the method of characteristics, applicable to one equation with one unknown function, to the system of PDE of \ref{eq:1.1}. The difference now is that there are many characteristics coming out of each point of the domain which the solution is being constructed on, therefore it is not possible to reduce it to systems of ODE. As described in the previous paragraph one way to gain control over these characteristics and the value of the solution, is to set bounds on them by discretizing the hyperplanes parallel to the initial condition hyperplane and later show that these bounds approach each other as the number of discretization hyperplanes goes to infinity. Also we derive explicit expressions for the locality condition and the Lipschitz constant of the solution of the PDE of Theorem \ref{eq:1.1} based on the constants of the problem as follows:
\begin{flalign}
\label{eq:1.2}
& \alpha  <  {1 \over  \exp \left( \theta(c_1) c_1  \alpha \right) n(m-1)L_C (L_I + 1/n) } \ \ , \ \ c_1 = n L_D - (m-1)L_C \\ & \label{eq:1.3} L_f = {(L_I + 1/n) \exp(c_1 \alpha) \over 1 - n (m-1) L_C \alpha (L_I + 1/n) \exp \left( \theta(c_1) c_1 \alpha \right)} -1/n \\ & \label{eq:1.4} L_{Ufs} = \max \left\{ L_f, M_{\Vert D \Vert} + L_f (m-1) M_{\Vert C \Vert} \right\}
\end{flalign}
$L_C$ and $L_D$ refer to the Lipschitz constants of the $C_{il}$ and $D_i$ functions on $P$, respectively. $\theta(c_1)$ is the step function. $L_I$ is the Lipschitz constant of the initial condition functions $I_i$ on $V$. $M_{\Vert D \Vert}$ and $M_{\Vert C \Vert}$ refer to a bound for $|D_{i}|$ and $|C_{il}|$ on $P$, respectively. The extent which, in general, the solution can be constructed in the $x_m$ direction above or below the initial condition hyperplane is given by the locality condition of \ref{eq:1.2}: $-\alpha \leq x_m - x_{0m} \leq + \alpha$. Also $\alpha \leq \bar{\alpha}$ with $ \bar{\alpha} = \min \{ a, (b - M_{\left\Vert I-y_0 \right\Vert})/M_{\Vert D \Vert} \} $ to make sure the domain and range of the solution lie within $P_1$ and $P_2$, respectively. With $L_f$ in relation \ref{eq:1.3} being the Lipschitz constant of the solution along the hyperplanes parallel to the initial condition hyperplane, $L_{Ufs}$ in relation \ref{eq:1.4} gives the total Lipschitz constant of the solution on its domain of construction.

One of the applications of Theorem \ref{theorem1.1} is in regard to hyperbolic quasilinear systems of first order PDE in two independent variables which, as an example, are used to describe the one dimensional space flow of fluids. These systems of PDE can be reduced to the PDE of Theorem \ref{theorem1.1} by differentiating the system, diagonalizing its coefficient matrix and performing a change of function variables, therefore Theorem \ref{theorem1.1} and the method which its solution is constructed (this is discussed in Section \ref{sec:3}) offer an alternative way, which is more direct and convenient especially for finding a numerical solution, as compared to other methods, e.g. iteration methods \cite{03}, for constructing the solution of hyperbolic quasilinear systems of first order PDE in two independent variables.

In order to illustrate the main idea of proving Theorem \ref{eq:1.1} in a simpler context, in Section \ref{sec:2} we present an alternative proof of the Picard-Lindelof theorem of ODE by setting upper and lower bounds on the value of the solution of the system of ODE: $y' = f(t,y), \ y(t_0) = y_0$, by discretizing the time interval $[t_0, t_0 + \alpha]$ into $2^N$ partitions at the $N$'th step
\begin{equation}
\label{eq:1.5}
y_{i,m}^{N,k}  \leq y_i(t_0 + k {\alpha / 2^N}) \leq y_{i,M}^{N,k}, \ \ \  k = 1, ..., 2^N, \ \ \  y_{i,M}^{N,0} = y_{i,m}^{N,0} = y_{0i}
\end{equation} 
 and find a recursion relation for $\Delta y^{N,k} \geq y_{i,M}^{N,k} - y_{i,m}^{N,k}$
 \begin{equation}
\label{eq:1.6}
\Delta y^{N,k} = \Delta y^{N,k-1}( 1 + nL_f \delta t)  + C \delta t ^2 + \epsilon \delta t, \ \ \ \Delta y^{N,0} = 0
\end{equation}
$\delta t  = \alpha /2^N$, $C$ a bounded constant, $L_f$ the Lipschitz constant of $f(t,y)$ and $\epsilon \rightarrow 0$ as ${\delta t \rightarrow 0}$. After solving relation \ref{eq:1.6} we find $\Delta y^{N,k} \sim 1 / 2^N + \epsilon $ therefore as $N \rightarrow \infty$, the upper and lower bounds for the solution in \ref{eq:1.5} approach each other, hence this gives a unique function for the solution to the system of ODE. We will see that this alternative way of proving the Picard-Lindelof theorem is more easily generalizable to the quasilinear system of PDE of \ref{eq:1.1}. Setting upper and lower bounds on the value of the solution enables us to have more control over the possible range of values of the solution and the bounds at the $N+1$'th step of partitioning naturally fall within the bounds at the $N$'th step of partitioning, therefore with denoting the set of possible ranges of values for the solution on the time interval at the $N$'th step of partitioning by $R^N$, these sets form a nested sequence $R^N \supseteq R^{N+1} \supseteq R^{N+2} \supseteq ... $ , hence in order to show that this nested sequence converges to the graph of a unique function for solution we only need to show that at the $N$'th step of the partitioning the difference between the upper and lower bounds of the solution is of order ${1 / 2^N}$. In the current methods which we make successive approximations to the solution without finding bounds for the solution, e.g. by making successive approximations to the solution from the integral equation of the system of ODE as in \cite{01} or considering the discretization of the system of ODE as when solving it numerically, in order to show convergence to a solution the difference between the approximations to the solution at the $N$'th step and the $N+1$'th step have to be found and finally show that the sequence of approximations to the solution at the $N$'th step converges uniformly to a solution. In these methods when the existence of the solution is proven one is not sure about its uniqueness and therefore a uniqueness proof has to be presented separately. In the method described above which we set bounds on the value of the solution the proof of the existence of the solution is not separate from proving the uniqueness of the solution, since in order to demonstrate existence it has to be shown that the bounds set on the solution at the $N$'th step form a nested sequence and approach each other as $N \rightarrow \infty$ which automatically shows uniqueness as well. This implies that this method is only applicable to when the conditions of the theorem are such that we obtain a unique solution (e.g. when $f(t,y)$ in the system of ODE above is Lipschitz), and it cannot be applied to show the existence of a solution only (e.g. it cannot be applied to when $f(t,y)$ is continuous).
\\
\\
\indent In Section \ref{sec:3} we prove Theorem \ref{theorem1.1}. We implement the same idea used in Section \ref{sec:2} and described in the paragraph after Theorem \ref{theorem1.1} to prove this result. A standard domain $S_+$ is defined as
\begin{equation}
\label{eq:1.7}
S_+ \equiv \{ x \in P_1 | \ 0 \leq x_m - x_{0 m} \leq \alpha , \  - \bar{a} + M_{C_l} ( x_m - x_{0 m} ) \leq x_l - x_{0 l}  \leq \bar{a} + m_{C_l} ( x_m - x_{0 m} )  \}
\end{equation}
and the solution is constructed on this domain. $m_{C_l}$ and $M_{C_l}$ refer to a lower and upper bound for $C_{il}$ for $i =1, ..., n$ on $P$, respectively. $\alpha > 0$ is chosen small enough. Similarly an $S_-$ domain can be defined for below the initial condition hyperplane \footnote[3]{A list of equivalent definitions for when constructing the solution on the $S_-$ domain is given in \ref{appen:A}.}. The domain between the initial condition hyperplane at $x_m = x_{0m}$ in $S_+$ and the hyperplane $x_m = x_{0m} + \alpha$ in $S_+$ is divided into $2^N$ equal partitions for $N = 0, \ 1, \ ...$ . The hyperplanes at $x_m = x_{0m} + k \alpha /2^N$ in $S_+$ are denoted by $V^{N,k}$ for $k = 1, ..., 2^N$ and $V^{N,0} \equiv V$. Upper and lower bound functions independent of the assumed solution are defined on $V^{N,k}$: $f^{N,k}_{i,M}: V^{N,k} \rightarrow \mathbb{R}$ and $f^{N,k}_{i,m}: V^{N,k} \rightarrow \mathbb{R}$ such that if $f(x)$ is a solution to \ref{eq:1.1} satisfying the initial condition then
\begin{equation}
\label{eq:1.8}
f^{N,k}_{i,m}(x) \leq f_i(x) \leq f^{N,k}_{i,M}(x) , \ \ \ \ \ \  x \in V^{N,k}
\end{equation}
and $f^{N,0}_{i,m} = f^{N,0}_{i,M} \equiv I_i$. Next in order to find a similar recursion relation as \ref{eq:1.6} for $\Delta f^{N,k} \geq f^{N,k}_{i,M}(x) - f^{N,k}_{i,m}(x) $, $x \in V^{N,k}$ we need to introduce the Lipschitz constants $L^{N,k}$ of $f^{N,k}_{i,M}$ and $f^{N,k}_{i,m}$ and to show that $\Delta f^{N,k} \rightarrow 0$ as $N \rightarrow \infty$ we need to show that these Lipschitz constants are bounded. This is done by finding a recursion relation for the Lipschitz constants in Section \ref{sec:3.1} and showing that they are locally (i.e. close enough to the initial condition hyperplane) bounded in Section \ref{sec:3.2}. The recursion relation for $\Delta f^{N,k}$, $L^{N,k}$ and a bound for the Lipschitz constants $L^{N,k} \leq L^{N,2^N}$ are given by
\begin{flalign}
\label{eq:1.9}
& \Delta f^{N,k} = \Delta f^{N,k-1} \left(1 + C_1 {\alpha \over 2^N} \right) + C_2 \left( {\alpha \over 2^N} \right)^2 \\ \label{eq:1.10}
& L^{N,k} \! \! = \! L^{N,{k-1}} \! \left(1 + \! (m \! - \! 1) L_C {\alpha \over 2^N} \! + \! n L_D {\alpha \over 2^N} \right) \! + \! n(m \! - \! 1)L_C { \alpha \over 2^N} \left( L^{N,{k-1}} \right) ^2 \!\!\! + \! L_D {\alpha \over 2^N} \\ \label{eq:1.11}
& L^{N, 2^N} \leq  {(L_I + 1/n) \exp(c_1 \alpha) \over 1 - n (m-1) L_C \alpha (L_I + 1/n) \exp \left( \theta(c_1) c_1 \alpha \right)} -1/n \equiv L_f 
\end{flalign}
with $\Delta f^{N,0} = 0$ and $L^{N,0} = L_I$. $C_1$ and $C_2$ are bounded constants.
If the locality condition of \ref{eq:1.2} is satisfied, it can be shown that $L^{N,k}$ are bounded for all $N$ and $k$, with their bound given by $L_f$ in relation \ref{eq:1.11}.

In \ref{appen:B} it is shown in detail that the bounds for the solution at the $N+1$ step of partitioning of $S_+$ lie within the bounds of the $N$ step of partitioning. Therefore with denoting the set of possible ranges of values of the solution on $S_+$ at the $N$ step of partitioning by $P^N_+$ we have $P^N_+ \supseteq P^{N+1}_+ \supseteq ...$ and $(x, f(x)) \in P^N_+$ for $x \in S_+$. Solving the recursion relation of \ref{eq:1.9} for $\Delta f^{N,k}$ we find $\Delta f^{N,k} \sim {1 / 2^N}$ hence $P^N_+$ converges to the graph of a unique function for the solution ($Ufs$) as $N \rightarrow \infty$.

Finally in Section \ref{sec:3.3} it is shown that the $Ufs$ obtained in the previous Subsections solves the system of PDE of Theorem \ref{theorem1.1} at its differentiable points subject to the initial condition. When the coefficients $C_{il}$, $D_i$ and the initial condition $I_i$ are $C^1$ in order to prove that $Ufs$ is $C^1$ on the hyperplanes $V^{N,k}$ the following functions are defined recursively
\begin{flalign}
\label{eq:1.12}
& f^{N,k}_i (x) = f^{N,k-1}_i \left( x - C_i \left( x^{\nu} , f^{N,k-1}(x^{\nu}) \right) \alpha/ 2^N \right) + D_i \left( x^{\nu} , f^{N,k-1}(x^{\nu}) \right) \alpha /2^N \\
& f^{N,0}_i \equiv I_i , \ x \in V^{N,k} , \ x^{\nu}  = x - \nu {\alpha \over 2^N} , \ \nu = \left( m_{C_l} + M_{C_l} \right) \hat{e}_l /2 + \hat{e}_m , C_i = (C_{i1}, ...,C_{im-1}, 1 ) \nonumber
\end{flalign}
the functions $f^{N,k}_i (x)$ are defined such that $f^{N,k}_{i,m}(x) \leq f^{N,k}_i (x) \leq f^{N,k}_{i,M}(x)$ for $x \in V^{N,k}$. A fixed $V^{N, k_N}$ is considered for $k_N \in \{ 1, ..., 2^N \}$ and $q = k_N / 2^N$ held fixed as $N \rightarrow \infty$. Based on the discussion above it is clear that the sequence of functions $f^{N,k_N}_i (x)$ converges uniformly to $Ufs(x)$ on $V^{N,k_N}$, furthermore it is shown that the sequence of their partial derivatives $\partial f^{N,k_N}_i / \partial x_l$ is bounded and equicontinuous, therefore there is a subsequence of their partial derivatives that converges uniformly. From this it is concluded that $Ufs(x)$ is $C^1$ on $V^{N,k_N}$, this is then easily generalized to all hyperplanes parallel to the initial condition hyperplane in $S_+$. Based on this fact it is then shown that $Ufs$ solves the system of PDE of \ref{eq:1.1} subject to the initial condition and is $C^1$ on $S_+$.

Note that relation \ref{eq:1.12} can be used to solve the system of PDE of \ref{eq:1.1} numerically on $S_+$. One might attempt to show that the discretized functions in \ref{eq:1.12} converge to the solution of the PDE of Theorem \ref{theorem1.1}. In this case one has to evaluate the difference between $f^{N,k}_i(x)$ and $f^{N+1,2k}_i(x)$ and show that this difference is of order $1/2^N$ uniformly on $V^{N,k}$ for $k =1, ..., 2^N$, this is also a possibility, however as mentioned earlier in the approach which we set bounds on the values of the solution things are more under control, therefore it is a more convenient and reliable method hence this will be the approach we consider in this work.

Section \ref{sec:4} discusses the generalizations and application of Theorem \ref{theorem1.1}. In Subsection \ref{sec:4.1} it is shown that the Lipschitz or $C^r$ dependence of the initial condition and coefficients $C_{il}$ and $D_i$ on parameters is inherited to the solution, Subsection \ref{sec:4.2} discusses the generalization of Theorem \ref{theorem1.1} to non-linear systems of PDE and in Subsection \ref{sec:4.3} the application of Theorem \ref{theorem1.1} in regard to quasilinear hyperbolic first order systems of PDE in two independent variables is briefly discussed.
\\
\\
\indent The generalization of Picard-Lindelof theorem/ the method of characteristics to systems of PDE is a result concerning the classical theory of partial differential equations which has remained unknown in the past centuries. As far as the author is concerned this result, in the form stated in Theorem \ref{theorem1.1} with minimal differentiability assumptions and explicit expressions for the locality condition and the Lipschitz constant of the solution, is not approachable using known methods or theorems and the only way is by direct construction of the solution. Here our main focus will be on proving this result and briefly discuss some of its generalizations and application but leave further investigations for future works.

\section{An alternative proof of the Picard-Lindelof theorem of ODE}
\label{sec:2}
In this Section we demonstrate the main idea in proving Theorem \ref{theorem1.1} in the simpler context of ordinary differential equations. Consider Picard-Lindelof theorem \footnote[4]{We make use of the maximum or infinity norm: $\Vert x \Vert_\infty = \underset{t}{\max} \ |x_t| $ and the 1-norm: $\Vert x \Vert_1 = \sum_t |x_t| $ throughout the paper.}:
\begin{theorem}[Picard-Lindelof theorem]
\label{theorem2.1}
Let $y, f \in \mathbb{R}^n$; $f(t,y)$ continuous on a parallelepiped
$R:  -a \leq t - t_0 \leq a,  \left\Vert y - y_0 \right\Vert_{\infty}  \leq b $ and Lipschitz continuous with respect to $y$. Let $M_{\Vert f \Vert}$ be a bound for $\Vert f(t, y) \Vert_{\infty}$ on R; $\alpha = \min \{ a, b/M_{\Vert f \Vert} \}$. Then
\begin{equation}
\label{eq:2.1}
y' = f(t,y), \ \ \ \ y(t_0) = y_0
\end{equation}
has a unique solution $y = y(t)$ on $[t_0 - \alpha, t_0 + \alpha]$.
\end{theorem}
The standard proofs of this theorem are textbook material \cite{01}. Here we present an alternative way to prove this theorem.
\begin{proof}[Alternative proof of Picard-Lindelof theorem]
\onehalfspacing
Lets assume the system of ODE \ref{eq:2.1} has a solution. We can integrate \ref{eq:2.1} for this solution to obtain
\begin{equation}
\label{eq:2.2}
y_i(t) = y_{0i} + \int^{t}_{t_0} f_i(\bar{t} ,y(\bar{t})) d\bar{t}, \ \ \ \ y(t_0) = y_0
\end{equation}
to first approximation the maximum and minimum values of this solution at $t = t_0 + \alpha$ are given by
\begin{equation}
\label{eq:2.3}
y_{i,m}^{0,1} \equiv y_{0i} + \alpha m_{f_i}^{0,1} \leq y_i(t_0 + \alpha) \leq y_{0i} + \alpha M_{f_i}^{0,1} \equiv y_{i,M}^{0,1} , \ \ \  i = 1, ..., n
\end{equation}
where $M_{f_i}^{0,1}$ and $m_{f_i}^{0,1}$ denote the maximum and minimum values of $f_i(t,y)$ in the region $R^{0,1} \equiv \left\{ (t,y) \left| 0 \leq t - t_0 \leq \alpha,  \Vert y - y_0 \Vert_{\infty} \leq M_{\Vert f \Vert} \alpha \right. \right\} $. Next we divide the interval $[t_0, t_0 + \alpha]$ in half. The maximum and minimum values of the solution at $t = t_0 + \alpha /2$ are given by
\begin{equation}
\label{eq:2.4}
y_{i,m}^{1,1} \equiv y_{0i} + m_{f_i}^{1,1} \alpha/2 \leq y_i(t_0 + \alpha/2) \leq y_{0i} + M_{f_i}^{1,1} \alpha/2 \equiv y_{i,M}^{1,1}
\end{equation}
where $M_{f_i}^{1,1}$ and $m_{f_i}^{1,1}$ are the maximum and minimum values of $f_i(t,y)$ in $R^{1,1} \equiv \{ (t,y) | 0 \leq t - t_0 \leq \alpha/2,  \Vert y - y_0 \Vert_{\infty} \leq M_{\Vert f \Vert} \alpha/2 \}$, respectively. Now we use the bounds in \eqref{eq:2.4} for the possible range of the solution at $t = t_0 + \alpha /2$ as a range of possible initial conditions at $t = t_0 + \alpha /2$ to find a better range of values for the solution at $t = t_0 + \alpha$. This is given by
\begin{equation}
\label{eq:2.5}
y_{i, m}^{1,2} \equiv y_{i,m}^{1,1} + m_{f_i}^{1,2} \alpha/2 \leq y_i(t_0 + \alpha) \leq y_{i,M}^{1,1} + M_{f_i}^{1,2} \alpha/2 \equiv y_{i, M}^{1,2}
\end{equation}
where $M_{f_i}^{1,2}$ and $m_{f_i}^{1,2}$ are the maximum and minimum values of $f_i(t,y)$ in $R^{1,2} \equiv \{ (t,y) \vert  \alpha/2 \leq t - t_0 \leq \alpha , \ y_{i, m}^{1,1} - M_{\Vert f \Vert}\alpha/2 \leq y_i \leq y_{i, M}^{1,1} + M_{\Vert f \Vert}\alpha/2  \}$, respectively. We continue this process by dividing the interval $[t_0, t_0 + \alpha]$ into $2^N$ equal intervals for $N = 0, 1, 2, ... $ and set bounds on the solution at $t = t_0 + k \alpha /2^N$ for $k =1, ..., 2^N$
\begin{equation}
\label{eq:2.5.1}
\begin{split}
& y_{i,m}^{N,k} \equiv y_{i,m}^{N,k-1} + m_{f_i}^{N,k} {\alpha / 2^N} \leq y_i(t_0 + k {\alpha / 2^N}) \leq y_{i, M}^{N,k-1} + M_{f_i}^{N,k} {\alpha / 2^N} \equiv y_{i, M}^{N,k} \\
& R^{N,k} \equiv \left\{ (t,y) \left|  (k-1) {\alpha \over 2^N} \leq t - t_0 \leq k {\alpha \over 2^N},   y^{N,k-1}_{i,m} - M_{\Vert f \Vert} {\alpha \over 2^N} \leq y_i \leq y^{N,k-1}_{i, M} + M_{\Vert f \Vert} {\alpha \over 2^N} \right. \right\}
\end{split}
\end{equation} 
 with $y_{i,m}^{N,0} = y_{i, M}^{N,0} = y_{0i}$ and $M_{f_i}^{N,k}$ and $m_{f_i}^{N,k}$ denoting the maximum and minimum values of $f_i(t,y)$ in $R^{N,k}$, respectively. From \ref{eq:2.5.1} it can be verified that the bounds for the solution at the $N+1$ step of the partitioning lie within the bounds at the $N$ step of the partitioning \footnote[5]{This can be seen as follows, with assuming $y_{i, m}^{N+1,2k-2} \geq y_{i, m}^{N,k-1}$, $y_{i, M}^{N+1,2k-2} \leq y_{i, M}^{N,k-1}$ (note that this is true for $k-1 = 0$) we have to show $y_{i, m}^{N+1,2k} \geq y_{i, m}^{N,k}$, $y_{i, M}^{N+1,2k} \leq y_{i, M}^{N,k}$,
\begin{equation}
\label{eq:2.5.2}
\begin{split}
 y_{i, M}^{N+1,2k} = y_{i, M}^{N+1,2k-1} + M^{N+1,2k}_{f_i} {\alpha \over 2^{N+1}} = y_{i, M}^{N+1,2k-2} + {1 \over 2} \left( M^{N+1,2k-1}_{f_i} + M^{N+1,2k}_{f_i} \right) {\alpha \over 2^{N}}
\end{split} 
\end{equation}
$M^{N+1,2k-1}_{f_i} \leq M^{N,k}_{f_i}$, $M^{N+1,2k}_{f_i} \leq M^{N,k}_{f_i}$ since $R^{N+1,2k-1} \subset R^{N,k}$ and $R^{N+1,2k} \subset R^{N,k}$ and by assumption $y_{i, M}^{N+1,2k-2} \leq y_{i, M}^{N,k-1}$, therefore this proves $y_{i, M}^{N+1,2k} \leq y_{i, M}^{N,k}$, the proof of $y_{i,m}^{N+1,2k} \geq y_{i,m}^{N,k}$ is similar.

It is clear that $R^{N+1,2k-1} \subset R^{N,k}$ since by assumption $y^{N+1,2k-2}_{i,m} \geq y^{N,k-1}_{i,m}$ and $y^{N+1,2k-2}_{i, M} \leq y^{N,k-1}_{i, M}$ and $R^{N+1,2k} \subset R^{N,k}$ since $y^{N+1,2k-1}_{i,m} = y^{N+1,2k-2}_{i, m} + m^{N+1,2k-1}_{f_i} \alpha /2^{N+1} \geq y^{N,k-1}_{i,m} - M_{\Vert f \Vert} \alpha/2^{N+1}$ and similarly $y^{N+1,2k-1}_{i,M} = y^{N+1,2k-2}_{i,M} + M^{N+1,2k-1}_{f_i} \alpha /2^{N+1} \leq y^{N,k-1}_{i,M} + M_{\Vert f \Vert} \alpha/2^{N+1}$ hence their $y_i$ range is a subset of the $y_i$ range of $R^{N,k}$ and their $t$ range is also clearly a subset of the $t$ range of $R^{N,k}$.}, therefore with defining $R^N \equiv \cup^{2^N}_{k=1} R^{N,k}$ we have, $R^N \supseteq R^{N+1} \supseteq R^{N+2} \supseteq ...$ and clearly based on how $R^N$ is defined we have $(t, y(t)) \in R^N$ for $t \in [t_0, t_0 + \alpha]$, hence if we show that as $N \rightarrow \infty$ , $y_{i,M}^{N,k} - y_{i,m}^{N,k} \rightarrow 0$ for $k = 1, ..., 2^N$ it can be concluded that the regions $R^N$ will shrink to a graph of a unique function for the solution to \ref{eq:2.1}. To show this consider the following recursion relation
\begin{equation}
\label{eq:2.6}
y_{i,M}^{N,k} - y_{i,m}^{N,k} = y_{i,M}^{N,k-1} - y_{i,m}^{N,k-1} + \left( M_{f_i}^{N,k} - m_{f_i}^{N,k} \right)\alpha/2^N
\end{equation}
by assumption the function $f$ satisfies the Lipschitz condition in its $y$ coordinates and being a continuous function defined on the compact region $R^{N,k}$ it assumes its maximum and minimum values $M_{f_i}^{N,k}$ and $m_{f_i}^{N,k}$ at certain points in $R^{N,k}$ therefore we have
\begin{equation}
\label{eq:2.7}
M_{f_i}^{N,k} - m_{f_i}^{N,k} \leq \sum_i \left\{ y_{i, M}^{N,k-1} + M_{\Vert f \Vert} \alpha/2^N - \left( y_{i, m}^{N,k-1} - M_{\Vert f \Vert} \alpha/2^N \right) \right\} L_f + \epsilon
\end{equation}
with $L_f$ being the Lipschitz constant of the function $f(t,y)$ with respect to $y$. Since the function $f(t,y)$ is continuous and it is defined on a compact set it is uniformly continuous therefore for any $\epsilon > 0$ there is a $\delta > 0$ (independent of $y$) such that if $|t-t'| < \delta$, $|f_i(t, y) - f_i(t', y)| < \epsilon$. Now we can choose $N$ large enough such that $\alpha / 2^N < \delta$. This defines the $\epsilon$ used in relation \ref{eq:2.7}. Using \ref{eq:2.7} we can derive an upperbound for \ref{eq:2.6}
\begin{equation}
\label{eq:2.8}
y_{i, M}^{N,k} - y_{i, m}^{N,k} \leq \Delta y^{N,k-1} + \sum_i \left\{ \Delta y^{N,k-1} L_f \alpha/2^N  + L_f M_{\Vert f \Vert} \alpha/2^{N-1} \alpha/2^N \right\} + \epsilon \alpha /2^N \equiv \Delta y^{N,k}
\end{equation}
with $y_{i, M}^{N,k-1} - y_{i, m}^{N,k-1} \leq \Delta y^{N,k-1} $. From \ref{eq:2.8} we have
\begin{equation}
\label{eq:2.9}
\Delta y^{N,k} = \Delta y^{N,k-1}( 1 + nL_f \delta t)  + C \delta t ^2 + \epsilon \delta t
\end{equation}
with $C \equiv 2n L_f M_{\Vert f \Vert}$ and $\delta t \equiv \alpha / 2^N$.  Solving \ref{eq:2.9} with noting that $\Delta y^{N,0} = 0$ we find
\begin{flalign}
\label{eq:2.10}
\Delta y^{N,k} & \! = \! ( C \delta t ^2 \! + \epsilon \delta t ) \{ 1 \! + \! ( 1 \! + \! n L_f \delta t) \! + \! ... \! + \! ( 1 + n L_f \delta t)^{k-1} \} \! = \! (C \delta t + \epsilon ) \{ ( 1 + n L_f \delta t)^k -1 \} / (nL_f) \nonumber \\ & \leq (C \delta t + \epsilon ) \left\{ \exp \left( n L_f \alpha {k / 2^N} \right) -1 \right\} / (nL_f)
\end{flalign}
From \ref{eq:2.10} it can be easily seen that as $N \rightarrow \infty$, $\Delta y^{N,k} \rightarrow 0$ for any $k = 1, ..., 2^N$ hence $R^N$ converges to a graph of a unique function for the solution ($Ufs$) on $[t_0, t_0 + \alpha]$. It can be shown that $Ufs$ indeed solves \ref{eq:2.1}:
\begin{flalign}
\label{eq:2.11}
Ufs_i(t + \Delta t) & - Ufs_i(t) = Ufs_i ( t + \Delta t ) - ( Ufs_i(t) + \Delta t f_i (t, Ufs(t))) + \Delta t f_i (t, Ufs(t)) \nonumber \\ & = O ({\Delta t}^2) + \epsilon O ( \Delta t ) + \Delta t f_i (t, Ufs(t))  \Longrightarrow Ufs'(t) = f(t, Ufs(t)))
\end{flalign}
with $\epsilon \rightarrow 0$ as $\Delta t \rightarrow 0$. The second equality follows from \ref{eq:2.9}, for $N=0$, $k = 1$, $\delta t = | \Delta t |$ \footnote[6]{$\Delta t$ can also be considered negative. Although relation \ref{eq:2.9} was derived by assuming we are moving in the positive time direction, clearly it is equivalently valid for when moving in the negative time direction (e.g. for when constructing the solution on $[-\alpha+ t_0, t_0]$ with $\delta t = \alpha /2^N > 0$ ). }, $\Delta y^{0,0} = 0$, with considering $y_0 = Ufs(t)$ as the initial condition at $t \in [0, \alpha]$ and noting that $y_{i,m}^{0,1} \leq Ufs_i(t + \Delta t) \leq y_{i,M}^{0,1}$ and $y_{i,m}^{0,1} \leq Ufs_i(t) + \Delta t f_i(t, Ufs(t)) \leq y_{i,M}^{0,1}$. It is clear that with a similar procedure we can construct a unique solution on $[- \alpha + t_0, t_0 ]$. \qed
\end{proof}

\section{A generalization of Picard-Lindelof theorem/ the method of characteristics to systems of PDE}
\label{sec:3}
In this Section we will apply the idea used in the previous Section for proving the Picard-Lindelof theorem to prove the theorem below.

\begin{theorem}[A generalization of Picard-Lindelof theorem/ the method of characteristics to systems of PDE]
\label{theorem3.1}
\onehalfspacing
Let $C_{il}, D_i : P \rightarrow \mathbb{R} , \ i = 1, ..., n , l =1, ..., m-1$, $m \geq 2$, be Lipschitz continuous or $C^r$ $( r\geq 1)$ functions defined on the parallelpiped $P \equiv P_1 \times P_2$ with $P_1 \equiv \{ x \in \mathbb{R}^m | \left\Vert x - x_0 \right\Vert_\infty \leq a, x_0 \in \mathbb{R}^m \}$ and $ P_2 \equiv \{ y \in \mathbb{R}^n |  \left\Vert y - y_0 \right\Vert_\infty \leq b, y_0 \in \mathbb{R}^n \}$. And let the Lipschitz continuous or $C^r$ initial condition function $I : V \rightarrow P_2$ for $V \equiv \{ x \in P_1 \vert x_m = x_{0m} , |x_l - x_{0l}| \leq \bar{a} \}$, $0 < \bar{a} \leq a$ and $M_{\left\Vert I-y_0 \right\Vert} < b$ with $M_{\left\Vert I-y_0 \right\Vert} \equiv \max \{\left\Vert I(u) - y_0 \right\Vert_\infty \vert u \in V \}$ be given. The following system of partial differential equations
\begin{flalign}
\label{eq:3.1}
C_{i1}(x,y) {\partial y_i \over \partial x_1} + ... + C_{im-1}(x,y) {\partial y_i \over \partial x_{m-1}} + {\partial y_i \over \partial x_m} = D_i(x,y)
\end{flalign}
has a unique Lipschitz continuous \footnote[7]{By Lipschitz continuous solution we mean a Lipschitz continuous function that solves the system of PDE \ref{eq:3.1} at its differentiable points. By Rademacher theorem a Lipschitz continuous function is differentiable almost everywhere.} or $C^r$ solution respectively, $f : B \rightarrow P_2$ for $V \subset B \subseteq P_1$, $B$ containing a neighbourhood of $V_{int}$, with $V_{int} \equiv \{ x \in P_1 \vert x_m = x_{0m} , |x_l - x_{0l}| < \bar{a} \}$ and $f$ reducing to the initial condition function $I$ on $V$, $f(u) = I(u)$ for $u \in V $.
\end{theorem}
\begin{proof}
\onehalfspacing
In a similar approach as the alternative proof of the Picard-Lindelof theorem presented in the previous Section we assume a solution exists and find bounds for this solution by dividing the domain along the $x_m$ direction into equal partitions and later show that these bounds approach each other as the number of partitions goes to infinity.

First we define a standard domain to construct the solution on. Let $M_{\Vert D \Vert}$ be a bound for $|D_i|$ on $P$ and $ \bar{\alpha} = \min \{ a, (b - M_{\left\Vert I-y_0 \right\Vert})/M_{\Vert D \Vert} \} $. Let $M_{C_l}$ and $m_{C_l}$ denote an upper and lower bound for $C_{il}$ for $i = 1, ..., n$ on $P$, respectively. We define the plus standard domain
\begin{equation}
\label{eq:3.3}
S_+ \equiv \{ x \in P_1 | \ 0 \leq x_m - x_{0 m} \leq \alpha , \  - \bar{a} + M_{C_l} ( x_m - x_{0 m} ) \leq  x_l - x_{0 l} \leq \bar{a} + m_{C_l} ( x_m - x_{0 m} )  \}
\end{equation}
with $\alpha > 0$ chosen sufficiently small as to satisfy the following conditions: i) a locality criteria (the first relation of \ref{eq:3.47}) to be derived in Subsection \ref{sec:3.2}, ii) $\alpha \leq \bar{\alpha}$, iii) to ensure the inequalities for $x_l$ in the definition of \ref{eq:3.3} are satisfied. Similarly an $S_-$ domain can be defined for below the hyperplane $V$ \footnote[8]{A list of equivalent definitions for when constructing the solution on the $S_-$ domain can be found in \ref{appen:A}.}. The standard domain $S_+$ is defined in a way as to ensure the following two properties. If a solution $f$ to \ref{eq:3.1} on $S_+$ exists satisfying the initial condition then: \\
i) $f (S_+) \subseteq P_2$.
\\
ii) Each characteristic curve $x^{(i)}$ of this solution lies within $S_+$ and connects with a point in the initial condition domain $V$.

In what follows we will construct a unique solution to \ref{eq:3.1} on $S_+$ that satisfies the initial condition. We will be using lots of notations and definitions. For a $p \in S_+$ after integrating \ref{eq:3.1} based on an assumed solution $f$ on $S_+$ that satisfies the initial condition we obtain the following integral and characteristic equations:
\begin{equation}
\label{eq:3.2}
\begin{split}
& f_i (p) = f_i (p^{(i)}_0) + \int^{p_m }_{x_{0m}} D_i(x^{(i)}(t) ,f(x^{(i)}(t))) dt, \ x^{(i)}(x_{0m}) = p^{(i)}_0 \in V, \  x^{(i)}(p_m) = p  \\
& {d x^{(i)}_j(t) \over dt } = C_{ij}(x^{(i)}(t),f(x^{(i)}(t))) , \ \ C_{i m} \equiv 1, \ \ j = 1, ..., m
\end{split}
\end{equation}
note that the parameter of the characteristic equations t, is the same as the $x_m$ coordinate. Next we divide $S_+$ along the $x_m$ direction into $2^N$ for $N = 0 ,1 , ... $ equal partitions and find upper and lower bounds for the value of the assumed solution $f$ at the intersection of these partitions in $S_+$, we have
\begin{equation}
\label{eq:3.4}
\begin{split}
& f^{N, k}_{i, m}(x) \equiv f^{N, k-1}_{i, m,V_{\text{res.}}}(x) + m^{N,k}_{D_i}(x) {\alpha \over 2^N} \leq f_i(x) \leq f^{N, k-1}_{i, M,V_{\text{res.}}}(x) + M^{N,k}_{D_i}(x) {\alpha \over 2^N} \equiv f^{N, k}_{i, M}(x) \\ & x \in V^{N,k}, V^{N,k} \equiv \{ z \in S_+ \vert z_m = x_{0m} + k \alpha / 2^N \}, \ V^{N,0} \equiv V, \ k =1, ...,2^N
\end{split}
\end{equation}
$f^{N, 0}_{i, M}(z) = f^{N, 0}_{i, m}(z) \equiv I_i(z)$ for $z \in V$. $f^{N, k}_{i, M}, f^{N, k}_{i, m}: V^{N,k} \rightarrow \mathbb{R} $ form upper and lower bounds for the value of the solution on $V^{N,k}$. $f^{N, k-1}_{i, M, V_{\text{res.}}}(x)$, $f^{N, k-1}_{i, m, V_{\text{res.}}}(x)$, $M^{N,k}_{D_i}(x)$ and $m^{N,k}_{D_i}(x)$ for $x \in V^{N,k}$ will be defined below. The bounds of relation \ref{eq:3.4} can be understood better in terms of the first relation of \ref{eq:3.2}. Writing this relation for $x \in V^{N,k}$ as the final point and $x^i \in V^{N,k-1}$ as the initial point we have
\begin{flalign}
\label{eq:3.4'}
f_i (x) = f_i (x^i ) + \int^{x_{0m} + k \alpha/2^N }_{x_{0m} + (k-1) \alpha/2^N} D_i(x^{(i)}(t) ,f(x^{(i)}(t))) dt, \ \ \  x^{(i)}(x_m) = x, \ \ x^{(i)}(x^i_m) = x^i
\end{flalign}
note that $x_m = x_{0m} + k \alpha/2^N$ for $x \in V^{N,k}$ and $x^i_m = x_{0m} + (k-1) \alpha/2^N$ for $x^i \in V^{N,k-1}$. The bounds of relation \ref{eq:3.4} are such that $f^{N, k-1}_{i, m,V_{\text{res.}}}(x) \leq f_i (x^i ) \leq f^{N, k-1}_{i,M,V_{\text{res.}}}(x) $ and $m^{N,k}_{D_i}(x) \leq D_i(x^{(i)}(t) ,f(x^{(i)}(t))) \leq M^{N,k}_{D_i}(x)$ for $ (k-1) \alpha/2^N \leq t - x_{0m} \leq k \alpha/2^N$. Next we give precise definitions for these bounds. We first define $f^{N, k-1}_{i,M, V}(x)$, $f^{N, k-1}_{i,m, V}(x)$, $M^{N,k}_{D_i}(x)$ and $m^{N,k}_{D_i}(x)$ for $x \in V^{N,k}$:
\begin{equation}
\label{eq:3.5}
\begin{split}
 & \hspace{-0.3cm} f^{N, k-1}_{i,M ,V}(x) \equiv \max \left\{ f^{N, k-1}_{i, M}(z) \left| z \in S^{N,k}_{+, x} \cap V^{N,k-1} \right. \right\},\\ & \hspace{-0.3cm} f^{N, k-1}_{i, m ,V}(x) \equiv \min \left\{ f^{N, k-1}_{i, m}(z) \left| z \in S^{N,k}_{+, x} \cap V^{N,k-1} \right. \right\}, \\
& \hspace{-0.3cm} M^{N,k}_{D_i}(x) \! \equiv \! \max \left\{ D_i(z,y) \left| (z,y) \in P^{N,k}_{+, x} \right. \right\}, m^{N,k}_{D_i}(x) \! \equiv \! \min \left\{ D_i(z,y) \left| (z,y) \in P^{N,k}_{+, x} \right. \right\}
\end{split}
\end{equation}
with $S^{N,k}_{+, x}$ and $P^{N,k}_{+, x}$ for $x \in V^{N,k}$ given by
\begin{equation}
\begin{split}
\label{eq:3.6}
& S^{N,k}_{+, x} \equiv \left\{ z \in S_+ \left| -{ \alpha / 2^N} \leq z_m - x_{m} \leq 0, \right.\right. \\ & \hspace{1 cm} \left. M_{C_l} ( z_m - x_{m}) \leq z_l - x_l \leq m_{C_l} ( z_m - x_m) \right\} \\
& P^{N,k}_{+, x} \equiv \{ (z,y) \left| z \in S^{N,k}_{+, x} ,\ f^{N, k-1}_{i,m, V}(x) \! - \! M_{\Vert D \Vert} {\alpha \over 2^N} \leq y_i \leq f^{N, k-1}_{i,M,V}(x) + M_{\Vert D \Vert} {\alpha \over 2^N}, i = 1, ..., n \right. \}
\end{split}
\end{equation}
for when the characteristic curves $x^{(i)}(t)$ of the assumed solution $f$ pass through a $x \in V^{N,k}$ for $i = 1, ...,n$, i.e. $x^{(i)}(x_m) = x$, $S^{N,k}_{+, x}$ is defined in a way as to ensure that $x^{(i)}(t) \in S^{N,k}_{+, x}$ for $- { \alpha / 2^N} \leq t - x_{m} \leq 0$ and $P^{N,k}_{+, x}$ is defined in a way as to ensure that $\left( x^{(i)}(t),f(x^{(i)}(t)) \right) \in P^{N,k}_{+, x}$ for $ - { \alpha / 2^N} \leq t - x_{m} \leq 0$. $f^{N, k-1}_{i, M, V_{\text{res.}}}(x)$ and $f^{N, k-1}_{i,m, V_{\text{res.}}}(x)$ for $x \in V^{N,k}$ are given by
\begin{flalign}
\label{eq:3.7}
& f^{N, k-1}_{i,M, V_{\text{res.}}}(x) \equiv \max \left\{ f^{N, k-1}_{i, M}(z) \left| z \in V^{N,k-1}_{\text{res.} , i, x} \right. \right\} \nonumber \\
& f^{N, k-1}_{i, m, V_{\text{res.}}}(x) \equiv \min \left\{ f^{N, k-1}_{i, m}(z) \left| z \in V^{N,k-1}_{\text{res.} , i, x} \right. \right\}  \\
& V^{N, k-1}_{\text{res.} ,i, x} \equiv \left\{ z \in S^{N,k}_{+, x} \cap V^{N,k-1} \left| - M^{N,k}_{C_{il}}(x) \alpha / 2^N  \leq z_l - x_l \leq - m^{N,k}_{C_{il}}(x) \alpha / 2^N \right. \right\} \nonumber
\end{flalign}
with $M^{N,k}_{C_{il}}(x)$ and $m^{N,k}_{C_{il}}(x)$ having similar definitions as $M^{N,k}_{D_i}(x)$ and $m^{N,k}_{D_i}(x)$ in relation \ref{eq:3.5} respectively with $D_i$ replaced by $C_{il}$
\begin{flalign}
\label{eq:3.8}
\hspace{-0.3cm} M^{N,k}_{C_{il}}(x) \equiv \max \left\{ C_{il}(z,y) \left| (z,y) \in P^{N,k}_{+, x} \right. \right\}, \ \ m^{N,k}_{C_{il}}(x) \equiv \min \left\{ C_{il}(z,y) \left| (z,y) \in P^{N,k}_{+, x} \right. \right\}
\end{flalign}

From the definitions above it can be verified that the bounds of relation \ref{eq:3.4} for the assumed solution are correct. For example the maximum of $f_i$ at a point $x \in V^{N,k}$ consists of the maximum value of $f_i$ in a region of $V^{N,k-1}$ which the characteristic curve of $f_i$ passing through this region has the possibility of passing through $x$, this region of $V^{N,k-1}$ is given by $V^{N, k-1}_{\text{res.} ,i , x}$ defined in \ref{eq:3.7} and the maximum value is given by $f^{N, k-1}_{i, M,V_{\text{res.}}}(x)$, plus the maximum value which $f_i$ can change when its characteristic curve passing through $S^{N,k}_{+, x}$ reaches $x$, this is given by $M^{N,k}_{D_i}(x) {\alpha / 2^N}$.

Also it can be verified that the bounds for the solution at the $N+1$ step of the partitioning lie within the bounds of the $N$ step. This is discussed in detail in \ref{appen:B}, therefore with defining $P^N_+ \equiv \cup^{2^N}_{k = 1} ( \cup_{x \in V^{N,k}} P^{N,k}_{+,x} ) $ we have $P^N_+ \supseteq P^{N+1}_+ \supseteq ... $ . From the definitions and relations above it is clear that the graph of the assumed solution on $S_+$ lies within the set $P^N_+$ at the $N$ step of partitioning, $(x, f(x)) \in P^N_+$ for $x \in S_+$, therefore in order to show that $P^N_+$ converges to the graph of a unique function for the solution ($Ufs$) we only need to show that $f^{N, k}_{i, M}(x) - f^{N, k}_{i, m}(x) \rightarrow 0$ as $N \rightarrow \infty$.

For this we will try to find a similar recursion relation as in \ref{eq:2.9} for $\Delta f^{N,k} \geq f^{N, k}_{i, M}(x) - f^{N, k}_{i, m}(x) $, $\forall x \in V^{N,k}$. Starting from \ref{eq:3.4} we have
\begin{equation}
\label{eq:3.9}
f^{N, k}_{i, M}(x) - f^{N, k}_{i, m}(x) = f^{N, k-1}_{i,M,V_{\text{res.}}}(x) - f^{N, k-1}_{i,m,V_\text{res.}}(x) + \left\{ M^{N,k}_{D_i}(x) - m^{N,k}_{D_i}(x) \right\} \alpha /2^N
\end{equation}
an upper bound for $M^{N,k}_{D_i}(x) - m^{N,k}_{D_i}(x)$ is given by
\begin{equation}
\label{eq:3.10}
\begin{split}
M^{N,k}_{D_i}(x) - m^{N,k}_{D_i}(x) \leq & \bigg{\{} \sum_i \left\{ f^{N, k-1}_{i,M, V}(x) - f^{N, k-1}_{i,m, V}(x) + 2 M_{\Vert D \Vert} {\alpha \over 2^N} \right\} \\ & + \sum_l (M_{C_l} - m_{C_l}) {\alpha \over 2^N} + {\alpha \over 2^N} \bigg{\}} L_D
\end{split}
\end{equation}
with $L_D$ a Lipschitz constant for the $D_i$ functions and the expression in brackets corresponds to an upperbound for the distance $\Vert p_1 - p_2 \Vert_1$ between any two points $p_1, p_2 \in P^{N,k}_{+,x}$ defined in \ref{eq:3.6}. We also need to find an upper bound for $f^{N, k-1}_{i,M,V}(x) - f^{N, k-1}_{i,m,V}(x)$ in \ref{eq:3.10}. For this we will assume the functions $f^{N, k}_{i,M}(x) $ and $f^{N, k}_{i,m}(x)$ are Lipschitz with Lipschitz constant $L^{N,k}$. We will show this to be true and derive a recursion relation for the Lipschitz constants $L^{N,k}$ in Subsection \ref{sec:3.1}. We have
\begin{equation}
\label{eq:3.11}
\begin{split}
f^{N, k-1}_{i,M,V}(x) - & f^{N, k-1}_{i,m,V}(x) = f^{N, k-1}_{i, M}(x^{i}_{\max}) - f^{N, k-1}_{i, M}(x^{i}_{\min}) + \left( f^{N, k-1}_{i,M}(x^{i}_{\min}) - f^{N, k-1}_{i,m}(x^{i}_{\min}) \right) \\ & \leq L^{N,k-1} \sum_l (M_{C_l} - m_{C_l}) \alpha /2^N + \left( f^{N, k-1}_{i,M}(x^{i}_{\min}) - f^{N, k-1}_{i,m}(x^{i}_{\min}) \right)
\end{split}
\end{equation}
with $x^{i}_{\max}$ and $x^{i}_{\min}$ denoting the points in $S^{N,k}_{+, x} \cap V^{N,k-1}$ which $f^{N, k-1}_{i, M}$ and $f^{N, k-1}_{i, m}$ assume their maximum and minimum values in $S^{N,k}_{+, x} \cap V^{N,k-1}$, respectively. Combining \ref{eq:3.10} and \ref{eq:3.11} we have
\begin{equation}
\label{eq:3.12}
\begin{split}
M^{N,k}_{D_i}(x) - m^{N,k}_{D_i}(x) \leq & \bigg{\{} n \Delta f^{N,k-1} +  L^{N,k-1} n \sum_l (M_{C_l} - m_{C_l}) {\alpha \over 2^N} \\ & + \bigg{(} 2 n M_{\Vert D \Vert} + \sum_l (M_{C_l} - m_{C_l}) + 1 \bigg{)} {\alpha \over 2^N} \bigg{\}} L_D
\end{split}
\end{equation}
with $\Delta f^{N,k-1}$ an upper bound for the following quantity
\begin{equation}
\label{eq:3.13}
f^{N, k-1}_{i, M}(z) - f^{N, k-1}_{i, m}(z) \leq \Delta f^{N,k-1} , \ \ \ \   \forall z \in V^{N,k-1}
\end{equation}
note that based on \ref{eq:3.13}, we can take $\Delta f^{N,0} = 0$ since we defined $f^{N, 0}_{i, M} = f^{N, 0}_{i, m} = I_i$.
Similarly we can obtain a bound for $f^{N, k-1}_{i, M,V_{\text{res.}}}(x) - f^{N, k-1}_{i, m,V_\text{res.}}(x)$ in \ref{eq:3.9}
\begin{flalign}
\label{eq:3.14}
 f^{N, k-1}_{i,M,V_{\text{res.}}}& (x) \! - \! f^{N, k-1}_{i,m,V_\text{res.}}(x) = \! f^{N, k-1}_{i, M}( z^{i}_{\max}) \! - \! f^{N, k-1}_{i, M}(z^{i}_{\min}) \! + \! \left( f^{N, k-1}_{i, M}(z^{i}_{\min}) \! - \! f^{N, k-1}_{i, m}(z^{i}_{\min})\right) \nonumber \\ & \leq L^{N,k-1} \! \sum_l \! \left( M^{N,k}_{C_{il}}(x) - m^{N,k}_{C_{il}}(x) \right) \! {\alpha \over 2^N} \! + \! \left( f^{N, k-1}_{i, M}(z^{i}_{\min}) \! - \! f^{N, k-1}_{i, m}(z^{i}_{\min})\right)
\end{flalign}
with $z^{ i }_{\max}$ and $z^{i}_{\min}$ denoting the points in $V^{N, k-1}_{\text{res.} ,i, x}$ which $f^{N, k-1}_{i,M}$ and $f^{N, k-1}_{i, m}$ assume their maximum and minimum values in $V^{N, k-1}_{\text{res.} ,i, x}$, respectively.
Similar to \ref{eq:3.12} we can obtain a bound for $M^{N,k}_{C_{il}}(x) - m^{N,k}_{C_{il}}(x)$. We have
\begin{equation}
\label{eq:3.15}
\begin{split}
M^{N,k}_{C_{il}}(x) - m^{N,k}_{C_{il}}(x) \leq & \bigg{\{}n \Delta f^{N,k-1} +  L^{N,k-1} n \sum_l (M_{C_l} - m_{C_l}) {\alpha \over 2^N} \\ & + \bigg{(} 2 n M_{\Vert D \Vert} + \sum_l (M_{C_l} - m_{C_l}) + 1 \bigg{)} {\alpha \over 2^N} \bigg{\}} L_C
\end{split}
\end{equation}
with $L_C$ a Lipschitz constant for the $C_{il}$ functions. Now using \ref{eq:3.12}, \ref{eq:3.14} and \ref{eq:3.15} we can find a bound for \ref{eq:3.9}, we have
\begin{flalign}
\label{eq:3.16}
& \Delta f^{N,k} \equiv \Delta f^{N,k-1} + \left( L_C (m-1) L^{N,k-1} + L_D \right) {\alpha \over 2^N} \bigg{\{} n \Delta f^{N,k-1} +  L^{N,k-1} n \sum_l (M_{C_l} - m_{C_l}) {\alpha \over 2^N} + \nonumber \\ & \bigg{(} 2 n M_{\Vert D \Vert} + \sum_l (M_{C_l} - m_{C_l}) + 1 \bigg{)} {\alpha \over 2^N} \bigg{\}} \geq  f^{N, k}_{i, M}(x) - f^{N, k}_{i, m}(x) , \ \ \forall x \in V^{N,k}
\end{flalign}
In Subsection \ref{sec:3.1} we will derive a recursion relation for the Lipschitz constants $L^{N,k}$ and in Subsection \ref{sec:3.2} we will show that they are locally (i.e. for a sufficiently small $\alpha$) bounded. With knowing this we can write \ref{eq:3.16} as
\begin{equation}
\label{eq:3.17}
\begin{split}
\Delta f^{N,k} & = \Delta f^{N,k-1} \left( 1 + C_1 \alpha/2^N \right) + C_2 \left( \alpha /2^N \right)^2
\end{split}
\end{equation}
with $C_1$ and $C_2$ being constants which bound the following quantities
\begin{flalign}
\label{eq:3.18}
& C_1 \geq n (m-1) L_C L^{N,k-1} + n L_D \\
& C_2 \geq \left( (m-1) L_C L^{N,k-1} + L_D \right) \bigg{\{} L^{N,k-1} n \sum_l ( M_{C_l} - m_{C_l} )  + 2 n M_{\Vert D \Vert} + \sum_l ( M_{C_l} - m_{C_l} ) + 1 \bigg{\}} \nonumber
\end{flalign}
\ref{eq:3.17} is the recursion relation similar to \ref{eq:2.9} we were looking for. For completeness we include the recursion relation for the Lipschitz constants to be derived in Subsection \ref{sec:3.1}, the locality criteria for $\alpha$ and a bound for the Lipschitz constants $L^{N,k}$, to be derived in Subsection \ref{sec:3.2}, and a Lipschitz constant for the unique function for the solution ($Ufs$) to Theorem \ref{theorem3.1} to be derived below, here
\begin{flalign}
\label{eq:3.20}
& \Delta f^{N,k} = \Delta f^{N,k-1} \left(1 + C_1 {\alpha \over 2^N} \right) + C_2 \left( {\alpha \over 2^N} \right)^2 \nonumber \\
& L^{N,k} = L^{N,{k-1}} \left(1 + (m-1) L_C {\alpha \over 2^N} +n L_D {\alpha \over 2^N} \right) + n(m-1)L_C { \alpha \over 2^N} \left( L^{N,{k-1}} \right) ^2 + L_D {\alpha \over 2^N} \nonumber \\
& \alpha  <  {1 \over  \exp \left( \theta(c_1) c_1  \alpha \right) n(m-1)L_C (L_I + 1/n) } \ \ , \ \ c_1 = n L_D - (m-1)L_C \\ & L^{N, 2^N} \leq  {(L_I + 1/n) \exp(c_1 \alpha) \over 1 - n (m-1) L_C \alpha (L_I + 1/n) \exp( \theta(c_1) c_1 \alpha)} -1/n \equiv L_f \nonumber \\ & L_{Ufs} = \max \left\{ L_f, M_{\Vert D \Vert} + L_f (m-1) M_{\Vert C \Vert} \right\} \nonumber
\end{flalign}
Relations \ref{eq:3.20} constitute the main relations of Theorem \ref{theorem3.1}. $L_I$ refers to the Lipschitz constant of the initial condition function $I$ and $\theta (c_1)$ the step function.

With knowing that the Lipschitz constants $L^{N,k}$ are locally bounded we can use the first relation in \ref{eq:3.20} to show that $\Delta f^{N,k} \rightarrow 0$ as $N \rightarrow \infty$ similar to the steps in relation \ref{eq:2.10}
\begin{flalign}
\label{eq:3.21}
 \Delta f^{N,k} & = C_2 ( \alpha /2^N )^2 \{1 + (1 + C_1 \alpha/2^N ) + ... + (1 + C_1 \alpha/2^N )^{k-1} \} \nonumber \\ & = C_2/C_1 \alpha /2^N \{ (1 + C_1 \alpha/2^N )^{k} - 1 \} \\ & \hspace{-1cm} \Longrightarrow \Delta f^{N,k} \leq C_2/C_1 \{ \exp(C_1 \alpha k /2^N) - 1 \} \alpha /2^N \nonumber
\end{flalign}
from relation \ref{eq:3.21} it is clear that $\Delta f^{N,k} \rightarrow 0$ as $N \rightarrow \infty$, hence $P^{N}_+$ converges to the graph of a unique function for the solution ($Ufs$) to Theorem \ref{theorem3.1}. We will prove in Subsection \ref{sec:3.3} that $Ufs$ indeed solves the PDE of Theorem \ref{eq:3.1} subject to the initial condition. Before moving on to the next Subsection we show that $Ufs$ is also Lipschitz in the $x_m$ direction. $L_f$ in relation \ref{eq:3.20} can be considered as the Lipschitz constant of $Ufs$ along the hyperplanes $x_m = \text{const.}$ in $S_+$ for $x_{0m} \leq \text{const.} \leq x_{0m} + \alpha$. Consider $V^{N,k_N}$ and $V^{N,k'_N}$ for $q = k_N /2^N$ and $q' = k'_N /2^N$ held fixed as $N \rightarrow \infty$ and $\Delta x_m = q' - q $ for $q' > q$. It can be easily seen that a bound for the difference $| f^{N,k'_N}_{i,M} (x + \hat{e}_m \Delta x_m) - f^{N,k_N}_{i, M}(x) |$ for $x \in V^{N,k_N}$ and $x + \hat{e}_m \Delta x_m \in V^{N,k'_N}$ is $M_{\Vert D \Vert} \Delta x_m + L_f (m-1) M_{\Vert C \Vert} \Delta x_m$, with $M_{\Vert C \Vert}$ being a bound for $| C_{il} |$ on $P$ and $\hat{e}_m $ the unit vector in the $x_m$ direction, hence $M_{\Vert D \Vert} + L_f M_{\Vert C \Vert} (m-1)$ can be considered as a Lipschitz constant for $Ufs$ in the $x_m$ direction. Therefore
\begin{flalign}
\label{eq:3.22}
L_{Ufs} = \max \{ L_f, M_{\Vert D \Vert} + L_f (m-1) M_{\Vert C \Vert} \} 
\end{flalign}
is a Lipschitz constant for $Ufs$ on $S_+$ (or $S_-$). Note that relations of \ref{eq:3.20} are equivalently valid for when constructing the solution on the $S_-$ domain with $\alpha > 0$ being the extent which, in general, the solution can be constructed below the initial condition hyperplane $V$. A list of the equivalent of the definitions used in this Section for when constructing the solution on the $S_-$ domain is given in \ref{appen:A}.

\subsection{A recursion relation for the Lipschitz constants $L^{N,{k}}$}
\label{sec:3.1}
In this Subsection we will obtain a recursion relation for the Lipschitz constants $L^{N,{k}}$ of the functions $f^{N, k}_{i, M} (x)$. A similar result will be reached if we work with the functions $f^{N, k}_{i, m} (x)$. Let $L^{N,0} = L_I$, with $L_I$ being the Lipschitz constant of the initial condition functions $I_i$. Take two separate points $p_1, p_2 \in V^{N,k}$. With assuming $L^{N,k-1}$ is known we would like to find an expression for $L^{N,k}$
\begin{equation}
\label{eq:3.23}
|f^{N, k}_{i, M} (p_1) - f^{N, k}_{i, M} (p_2)| \leq L^{N,k} \sum_l | p_{1 l } - p_{2 l} | , \ \ \  l =1, ..., m-1
\end{equation}
For this we will make use of the following Lemma:
\begin{lemma}
\onehalfspacing
\label{lemma3.1}
Let $g: W \subseteq \mathbb{R}^n \rightarrow \mathbb{R}$ be a Lipschitz continuous function with Lipschitz constant $L_g$ for the 1-norm. $W_1, W_2 \subseteq W$ be compact sets and consider $d$ with the following characteristics:
\begin{equation}
  \forall w_1 \in W_1, \exists w_2 \in W_2 \! : \! \Vert w_1 - w_2 \Vert_1 \! \leq d, \ \text{and} \ \text{vice versa:} \ \forall w_2 \! \in \! W_2, \exists w_1 \! \in \! W_1 \! : \! \Vert w_1 - w_2 \Vert_1 \! \leq \! d \nonumber
\end{equation}
then we have the following relations: $|M_g(W_1) - M_g(W_2)| \leq L_g d$ and $|m_g(W_1) - m_g(W_2)| \leq L_g d$. Where $M_g(W_r)$ and $m_g(W_r)$ denote the maximum and minimum values of $g$ in $W_r$ for $r = 1, 2$, respectively.
\end{lemma}
\begin{proof}
By the assumption of compactness of $W_r$ and continuity of $g$ there exists $w_r \in W_r$ such that $g(w_r) = M_g(W_r)$ for $r = 1,2$. By assumption of the lemma there is a $y_2 \in W_2$ such that $\Vert w_1 - y_2 \Vert_1 \leq d$ so we have $| g(w_1) - g(y_2) | \leq L_g d$ and since $g(y_2) \leq g(w_2)$ we have : $g(w_2) + L_g d \geq g(w_1)$ and similarly it can be concluded $g(w_1) + L_g d \geq g(w_2)$ which proves $|M_g(W_1) - M_g(W_2)| \leq L_g d$. Similarly it can be concluded that $|m_g(W_1) - m_g(W_2)| \leq L_g d$.
\end{proof}
\textbf{Note:} Consider $ B_1 \equiv \prod^n_{h = 1} [a_h,b_h] , B_2 \equiv \prod^n_{h = 1} [c_h,d_h] \subset \mathbb{R}^n $. Then $d = \sum^n_{h = 1} \max \{ |a_h - c_h|, |b_h - d_h| \} $ has the characteristics of the distance $d$ in Lemma \ref{lemma3.1} with respect to the subsets $B_1$ and $B_2$.

From \ref{eq:3.4}
\begin{equation}
\label{eq:3.24}
f^{N, k}_{i, M} (p_1) - f^{N, k}_{i, M} (p_2) = f^{N, k-1}_{i, M,V_{\text{res.}}}(p_1) - f^{N, k-1}_{i, M,V_{\text{res.}}}(p_2) + \left( M^{N,k}_{D_i}(p_1) - M^{N,k}_{D_i}(p_2) \right) {\alpha \over 2^N}
\end{equation}
assuming $d_1$ has the characteristics of the distance $d$ in Lemma \ref{lemma3.1} for the two sets $V^{N, k-1}_{\text{res.}, i, p_1 }$ and $V^{N, k-1}_{\text{res.}, i , p_2}$ we have
\begin{equation}
\label{eq:3.25}
|f^{N, k-1}_{i,M,V_{\text{res.}}}(p_1) - f^{N, k-1}_{i,M,V_{\text{res.}}}(p_2)| \leq L^{N,k-1} d_1
\end{equation}
based on the definitions of $V^{N, k-1}_{\text{res.} , i , p_r}$ for $r = 1,2$ in \ref{eq:3.7} and the \textbf{Note} after Lemma \ref{lemma3.1} we can find an expression for $d_1$
\begin{flalign}
\label{eq:3.26}
d_1 \! \geq \!\! \sum_l \! \max \! \left\{ \left\vert p_{1l} \! - \! p_{2l} \! + \! \left( \! M^{N,k}_{C_{il}}(p_2) \! - \! M^{N,k}_{C_{il}}(p_1) \! \right) \! {\alpha \over 2^N} \right\vert \! , \left\vert p_{1l} \! - \! p_{2l} \! + \! \left( \! m^{N,k}_{C_{il}}(p_2) \! - \! m^{N,k}_{C_{il}}(p_1) \! \right) \! {\alpha \over 2^N} \right\vert  \right\}
\end{flalign}
a bound for $\left\vert M^{N,k}_{C_{il}}(p_1) - M^{N,k}_{C_{il}}(p_2) \right\vert$ or $\left\vert m^{N,k}_{C_{il}}(p_1) - m^{N,k}_{C_{il}}(p_2) \right\vert$ is given by
\begin{flalign}
\label{eq:3.27}
\left\vert M^{N,k}_{C_{il}}(p_1) - M^{N,k}_{C_{il}}(p_2) \right\vert \leq L_C d_2 , \ \ \ \left\vert m^{N,k}_{C_{il}}(p_1) - m^{N,k}_{C_{il}}(p_2) \right\vert \leq L_C d_2
\end{flalign}
with $d_2$ having the characteristics of the distance $d$ in Lemma \ref{lemma3.1} for the two sets $P^{N,k}_{+, p_1}$ and $P^{N,k}_{+, p_2}$. Based on the definitions of $P^{N,k}_{+, p_r}$ for $r = 1, 2$ in \ref{eq:3.6} and the \textbf{Note} after Lemma \ref{lemma3.1} we can find an expression for $d_2$
\begin{equation}
\label{eq:3.28}
d_2 \geq \sum_l \left\vert p_{1l} - p_{2l} \right\vert + \sum_i \max \left\{ \left| f^{N,k-1}_{i,M,V}(p_1) - f^{N,k-1}_{i,M,V}(p_2) \right|, \left| f^{N,k-1}_{i,m,V}(p_1) - f^{N,k-1}_{i,m,V}(p_2) \right| \right\} 
\end{equation}
a bound for $\left| f^{N,k-1}_{i,M,V}(p_1) - f^{N,k-1}_{i,M,V}(p_2) \right|$ or $\left| f^{N,k-1}_{i,m,V}(p_1) - f^{N,k-1}_{i,m,V}(p_2) \right|$ is given by
\begin{equation}
\label{eq:3.29}
\begin{split}
\left| f^{N,k-1}_{i,M,V}(p_1) - f^{N,k-1}_{i,M,V}(p_2) \right| \leq L^{N,k-1} \sum_l \left\vert p_{1l} - p_{2l} \right\vert \\ \left| f^{N,k-1}_{i,m,V}(p_1) - f^{N,k-1}_{i,m,V}(p_2) \right| \leq L^{N,k-1} \sum_l \left\vert p_{1l} - p_{2l} \right\vert
\end{split}
\end{equation}
from the definitions of \ref{eq:3.5} and \ref{eq:3.6} it can be verified that $\sum_l \left\vert p_{1l} - p_{2l} \right\vert$ has the characteristics of the distance $d$ in Lemma \ref{lemma3.1} for the two sets $S^{N,k}_{+, p_1} \cap V^{N,k-1}$ and $S^{N,k}_{+, p_2} \cap V^{N,k-1}$.

From \ref{eq:3.28} and \ref{eq:3.29}, $d_2$ is given by
\begin{equation}
\label{eq:3.30}
d_2 = \sum_l \left\vert p_{1l} - p_{2l} \right\vert \left( 1 + n L^{N,k-1} \right)
\end{equation}
and from \ref{eq:3.26} and \ref{eq:3.27} $d_1$ is given by
\begin{equation}
\label{eq:3.31}
 d_1 = \sum_l \left\vert p_{1l} - p_{2l} \right\vert + (m-1) L_C {\alpha \over 2^N} d_2
\end{equation}
Similarly a bound for $\left| M^{N,k}_{D_i}(p_1) - M^{N,k}_{D_i}(p_2) \right|$ is given by
\begin{flalign}
\label{eq:3.32}
\left| M^{N,k}_{D_i}(p_1) - M^{N,k}_{D_i}(p_2) \right| \leq L_D d_2
\end{flalign}
From \ref{eq:3.25}, \ref{eq:3.30}, \ref{eq:3.31} and \ref{eq:3.32} we obtain a bound for \ref{eq:3.24}
\begin{equation}
\label{eq:3.33}
\begin{split}
\left| f^{N, k}_{i, M} (p_1) - f^{N, k}_{i, M} (p_2) \right| \leq & \left\{ L^{N,k-1} \left(1 + (m-1) L_C {\alpha \over 2^N} + (m-1) L_C {\alpha \over 2^N} n L^{N,k-1} \right) \right. \\ & \left. + L_D {\alpha \over 2^N} \left( 1 + n L^{N,k-1} \right) \right\} \sum_l \left\vert p_{1l} - p_{2l} \right\vert
\end{split}
\end{equation}
Comparing \ref{eq:3.23} and \ref{eq:3.33} we find an expression for $L^{N,k}$
\begin{equation}
\label{eq:3.34}
L^{N,k}\! = \! L^{N,k-1} \! \left( 1 + (m-1) L_C {\alpha \over 2^N} + nL_D {\alpha \over 2^N} \right) \! + n(m- 1)L_C {\alpha \over 2^N} \! \left( L^{N,k-1} \right)^2 \!\! + L_D {\alpha \over 2^N}
\end{equation}

\subsection{Local boundedness of $L^{N,k}$}
\label{sec:3.2}
The nonlinear term $\sim \left( L^{N,k-1} \right) ^2$ in \ref{eq:3.34} is the term that can lead to an unbounded increase of the Lipschitz constants $L^{N,k}$, but if the coefficient of this term $( \sim \! n(m-1)L_C \alpha )$ is small enough we expect to be able to show that the Lipschitz constants are bounded. We first rewrite \ref{eq:3.34} in a simpler form
\begin{equation}
\label{eq:3.35}
\begin{split}
& b_k  = b_{k-1} \gamma + ( b_{k-1} )^2, \ \ \ b_k \equiv c_2 {\alpha \over 2^N} ( L^{N,k} + {1 / n}) , \ \ \ k =1, ..., 2^N \\
& \gamma = \left(1 + c_1 \alpha / 2^N \right), \ \ c_1 = nL_D - (m-1)L_C, \ \ c_2 = n(m-1)L_C
\end{split}
\end{equation}
For convenience we have suppressed the index $N$ in $b_k$. Note that for $c_1 > 0$, $\gamma > 1$ but for $c_1 < 0$ and a sufficiently large $N$, $0 < \gamma < 1$ with $\gamma \rightarrow 1$ as $N \rightarrow \infty$. The first few terms of the sequence $b_k$ read
\begin{equation}
\label{eq:3.36}
\begin{split}
& b_1 = b_0 \gamma + { b_0 } ^2 \\ & b_2 = \gamma^2 b_0 + (\gamma + \gamma^2) { b_0 } ^2 + 2 \gamma { b_0 }^3 + { b_0 }^4
\end{split}
\end{equation}
From \ref{eq:3.35} and \ref{eq:3.36} it is clear that $b_k$ is a polynomial of degree $2^k$ in $b_0$
\begin{flalign}
\label{eq:3.37}
b_{k'} = \sum^{2^{k'}}_{h = 1} C^h_{k'} (\gamma) b_0^h , \ \ \ \  C^h_{k'} = 0 \ \text{for} \ h > 2^{k'} \ \text{or} \ h < 1, \ k' \in \mathbb{N} \cup \{ 0 \} 
\end{flalign}
with $C^h_{k'} (\gamma)$ a polynomial in $\gamma$. To show the local boundedness of $L^{N,k}$ we need to find a bound for the coefficients $C^h_k$. For this insert $b_{k-1}$ from relation \ref{eq:3.37} into relation \ref{eq:3.35} to obtain
\begin{equation}
\label{eq:3.38}
b_k = \gamma C^h_{k-1} b_0^h + \left(C^h_{k-1} b_0^h \right)^2 = \gamma C^h_{k-1} b_0^h + \sum^{2^{k}}_{h_1 = 2} \sum^{h_1 -1}_{h_2 = 1} C^{h_1 - h_2}_{k-1}C^{h_2}_{k-1} b_0^{h_1}
\end{equation}
summation over $h$ is implicit. From \ref{eq:3.38} a recursion relation for the coefficients $C^h_k$ can be derived
\begin{equation}
\label{eq:3.39}
\begin{split}
& C^h_k = \gamma C^h_{k-1} + 2 C^{h-1}_{k-1} C^1_{k-1} + ... + 2 C^{{h / 2} +1}_{k-1} C^{{h / 2} -1}_{k-1} + \left( C^{{h / 2}}_{k-1} \right)^2, \ \ \  \text{if} \ h \ \text{is} \ \text{even} \\
& C^h_k = \gamma C^h_{k-1} + 2 C^{h-1}_{k-1} C^1_{k-1} + ... + 2 C^{(h +1)/2}_{k-1} C^{(h -1)/2}_{k-1}, \ \ \  \text{if} \ h \ \text{is} \ \text{odd}
\end{split}
\end{equation}
In what follows we will show that the coefficients $C^h_k$ are bounded by the inequalities below
\begin{equation}
\label{eq:3.40}
\begin{split}
& C^h_k \leq k^{h-1} \gamma^{k h}, \ \ \ \ \ \  \gamma \geq 1 \\ & C^h_k \leq k^{h-1} \gamma ^{k-1} \ , \ \ \ {1 / 2} \leq \gamma < 1
\end{split}
\end{equation}
one might be able to improve the bounds in \ref{eq:3.40} and accordingly improve the bounds of relation \ref{eq:3.47} by a more careful study of the coefficients $C^h_k$. But these bounds suffice to capture the main features of a locality condition for $\alpha$.

From relation \ref{eq:3.35} and \ref{eq:3.36} it can be verified that $C^1_k = \gamma^k$ and $C^2_k = \sum^{2k -2}_{h = k-1} \gamma^h$ , satisfying the inequalities of \ref{eq:3.40}. So assuming $C^{h'}_k \leq k^{h'-1} \gamma^{kh'}$ holds for $1 \leq h' < h$ lets try to prove $C^{h}_k \leq k^{h-1} \gamma^{kh}$ for $\gamma \geq 1$ and for $h \geq 3$. Applying this to \ref{eq:3.39} we have
\begin{equation}
\label{eq:3.41}
\begin{split}
& C^h_k \leq \gamma C^h_{k-1} + 2 (h / 2 -1) (k-1)^{h-2} \gamma^{h(k-1)} + (k-1)^{h-2} \gamma^{h(k-1)}, \ \ \  \text{if} \ h \ \text{is} \ \text{even} \\ & C^h_k \leq \gamma C^h_{k-1} + 2 {(h - 1) \over 2} (k-1)^{h-2} \gamma^{h(k-1)}, \ \ \  \text{if} \ h \ \text{is} \ \text{odd}
\end{split}
\end{equation}
so in both cases we obtain
\begin{flalign}
\label{eq:3.42}
C^h_k \leq \gamma C^h_{k-1} + (h - 1) (k-1)^{h-2} \gamma^{h(k-1)}
\end{flalign}
applying this inequality to $C^h_{k-1} , C^h_{k-2}, ...$ we obtain
\begin{flalign}
\label{eq:3.43}
C^h_k & \leq \gamma^{k} C^h_{0}+ (h-1)0^{(h-2)} \gamma^{0+k-1} + (h-1)1^{(h-2)} \gamma^{h+k-2} + ... + (h - 1) (k-2)^{h-2} \gamma^{h(k-2) +1} \nonumber \\ & + (h - 1) (k-1)^{h-2} \gamma^{h(k-1)} \leq 0 + \gamma^{kh} \int^{k}_0 (h-1) x^{h-2} dx = k^{h-1} \gamma^{kh}
\end{flalign}
note that $C^h_{0} = 0$ for $h \geq 3$. We also used the fact that $\gamma^{kh} \geq \gamma^{h(k -1- r) + r}$ for $\gamma \geq 1$, $r =0, ..., k-1$, $k =1, ..., 2^N$ and $h \geq 3$ in the above relation.

Similarly if we assume $C^{h'}_k \leq k^{h'-1} \gamma^{k-1}$ holds for $1 \leq h' < h$, it is possible to prove that $C^{h}_k \leq k^{h-1} \gamma^{k-1}$ for $1/2 \leq \gamma < 1$ and $h \geq 3$. Applying this to \ref{eq:3.39} for both even and odd cases we obtain
\begin{flalign}
\label{eq:3.44}
C^h_k \leq \gamma C^h_{k-1} + (h - 1) (k-1)^{h-2} \gamma^{2(k-2)}
\end{flalign}
applying this inequality to $C^h_{k-1} , C^h_{k-2}, ...$ we have
\begin{flalign}
\label{eq:3.45}
& C^h_k \leq \! \gamma^{k} C^h_{0} \!+ (h \!- \! 1) \{ 0^{(h-2)} \gamma^{-2 + k-1} + \! 1^{(h-2)} \gamma^{0+k-2} \! + 2^{(h-2)} \gamma^{2+k-3} + ... + (k \! - \! 2)^{h-2} \gamma^{2(k-3) +1} \} \nonumber \\ & \! + \! (h \! - \! 1) (k\! - \!1)^{h-2} \gamma^{2(k-2)} \! \leq 0 + \! \gamma^{k-1}(h\!- \!1) \Big{\{} \! \int^{2}_0 \!\! x^{h-2} dx + \!\!\! \int^{k}_2 \! \! x^{h-2} dx \Big{\}} \! = \! k^{h-1} \gamma^{k-1}
\end{flalign}
we used the fact that $\gamma^{\tilde{k}-1} \geq \gamma^{2(\tilde{k}-2 - r) + r}$,$r =0, ..., \tilde{k}-3$, $\tilde{k} =3, ..., 2^N$ and $\int^{2}_0 x^{h-2} dx \geq 1/\gamma$ for $1/2 \leq \gamma < 1$ in the above relation. Hence the inequalities of \ref{eq:3.40} are proven. Applying \ref{eq:3.40} to \ref{eq:3.37} for $k = 2^N$ we find
\begin{equation}
\label{eq:3.46}
\begin{split}
& c_2 {\alpha} (L^{N, 2^N} + {1 / n}) = 2^N b_{2^N} \leq \sum^{2^{2^N}}_{h = 1} ( 2^N b_0 \gamma^{2^N} )^h < {c_2 \alpha (L_I + 1/n) \exp(c_1 \alpha) \over 1 - c_2 \alpha (L_I + 1/n) \exp(c_1 \alpha)} , \ \ \ c_1 \geq 0 \\ & c_2 {\alpha} (L^{N, 2^N} \!\!\! + \! {1 / n}) \! = \! 2^N b_{2^N} \! \leq \! \sum^{2^{2^N}}_{h = 1} \! {\gamma^{2^N} \over \gamma }  ( 2^N b_0 )^h \! < \! {1 \over \gamma} {\exp(c_1 \alpha) c_2 \alpha (L_I + 1/n) \over 1 - c_2 \alpha (L_I + 1/n) } ,  c_1 \! < \! 0 , 1/2 \leq \gamma < 1
\end{split}
\end{equation}
with $L^{N,0} = L_I$ the Lipschitz constant of the initial condition function $I$. We used $\gamma^{2^N} = (1 + (\alpha c_1) /2^N )^{2^N} \leq \exp (c_1 \alpha) $ in the above relations and assumed $2^N b_0 \exp (c_1 \alpha) < 1$ in the first relation and $2^N b_0 < 1$ in the second relation of \ref{eq:3.46}. From these assumptions and \ref{eq:3.46} we can find a locality condition for $\alpha$ and a bound for the Lipschitz constants $L^{N,k} \leq L^{N,2^N}$ \footnote[9]{We have dropped the $1 / \gamma$ factor on the righthand side of the second relation of \ref{eq:3.46} as $\gamma \rightarrow 1$ for $N \rightarrow \infty$. But now since $L^{N,2^N}$ is an increasing function of $N$ and the second relation of \ref{eq:3.47} is true in the limit of $N \rightarrow \infty$ then it must be true for all $N \in \{ 0 \} \cup \mathbb{N}$. To see how $L^{N,2^N}$ is an increasing function of $N$ consider $L^{N,k} = L^{N,k-1} \left( 1 + {e_1 / 2^N} \right) + (e_2 / 2^N) \left( L^{N,k-1} \right) ^2 + {e_3 / 2^N}$ from \ref{eq:3.34} with $e_1,e_2,e_3 \geq 0$. It suffices to show $L^{N+1,2k} \geq L^{N,k}$ for $k = 1, ..., 2^N$. Note that $L^{N+1,0} = L^{N,0} = L_I$, therefore lets assume $L^{N+1,2(k-1)} \geq L^{N,k-1}$ and try to prove $L^{N+1,2k} \geq L^{N,k}$. We have $L^{N+1,2k-1} = L^{N+1,2k-2} \left( 1 + {e_1 / 2^{N+1}} \right) + (e_2 / 2^{N+1}) \left( L^{N+1,2k-2} \right) ^2 + {e_3 / 2^{N+1}}$ and $L^{N+1,2k} = L^{N+1,2k-1} \left( 1 + {e_1 / 2^{N+1}} \right) + (e_2 / 2^{N+1}) \left( L^{N+1,2k-1} \right) ^2 + {e_3 / 2^{N+1}}$ $\Rightarrow$ $L^{N+1,2k} = L^{N+1,2k-2} \left( 1 + {e_1 / 2^{N}} \right) + (e_2 / 2^{N}) \left( L^{N+1,2k-2} \right) ^2 + {e_3 / 2^{N}} + terms \ greater \ than \ or \ equal \ to \ zero $. This proves $L^{N+1,2k} \geq L^{N,k}$.}
\begin{equation}
\label{eq:3.47}
\begin{split}
& \alpha  <  {1 \over  \exp \left( \theta(c_1) c_1  \alpha \right) n(m-1)L_C (L_I + 1/n) } \ \ , \ \ c_1 = n L_D - (m-1)L_C \\ &  L^{N, 2^N} \leq  {(L_I + 1/n) \exp(c_1 \alpha) \over 1 - n (m-1) L_C \alpha (L_I + 1/n) \exp( \theta(c_1) c_1 \alpha)} -1/n \equiv L_f
\end{split}
\end{equation}
with $\theta(c_!)$ the step function.

\subsection{Unique function for solution ($Ufs$) solves Theorem \ref{theorem3.1}}
\label{sec:3.3}
In this Subsection we will show that the $Ufs$ obtained in the previous Subsections is the solution of the system of PDE of Theorem \ref{theorem3.1}. With the Lipschitz condition for the initial condition and the coefficients $C_{il}$ and $D_i$, $Ufs$ is Lipschitz. Due to Radamechar theorem it is differentiable almost everywhere. Here we will show that $Ufs$ solves the system of PDE at its differentiable points. Consider two hyperplanes in $S_+$: $V_{\beta} = \{  z \in S_+ | z_m = \beta , x_{0m} \leq \beta \leq x_{0m} + \alpha \} $ and $V_{\beta + \delta \beta} = \{  z \in S_+ | z_m = \beta + \delta \beta , x_{0m} \leq \beta + \delta \beta \leq x_{0m} + \alpha \} $ for $\delta \beta > 0$. Define the function $g$ for $x - \hat{e}_m \delta \beta \in V_{\beta}$ and $x \in V_{\beta + \delta \beta}$
\begin{flalign}
\label{eq:3.48}
& g_i(x) \equiv Ufs_i \left( x - C_{i}(x^{\nu} , Ufs \left( x^{\nu}) \right) \delta \beta \right) + D_i(x^{\nu},Ufs(x^{\nu}))\delta \beta ,  \\ & C_i = (C_{i1}, ...,C_{im-1},1), \ x^{\nu} = x - \nu \delta \beta , \ \text{with} \ \  \nu = (m_{C_l} + M_{C_l}) \hat{e}_l /2 + \hat{e}_m \nonumber
\end{flalign}
$\hat{e}_j$ is the unit $m$-vector in the $x_j$ direction. Similar to before we can take $V_\beta$ as the initial condition hyperplane and $V_{\beta + \delta \beta}$ as the final hyperplane, but we will not partition the space in between, instead we take the limit $\delta \beta \rightarrow 0$. Based on how $g_i(x)$ is defined it can be seen to lie within the upper and lower bounds for the solution\footnote[10]{e.g. it can be verified that $x^{\nu} \in S^{0,1}_{+, x}$, $(x^{\nu}, Ufs(x^{\nu})) \in P^{0,1}_{+, x}$ therefore $m^{0,1}_{D_i}(x ) \leq D_i(x^{\nu},Ufs(x^{\nu})) \leq  M^{0,1}_{D_i}(x)$ and $-M^{0,1}_{C_{il}}(x) \leq - C_{il}(x^{\nu},Ufs(x^{\nu})) \leq -m^{0,1}_{C_{il}}(x )$ hence $(x - C_{i}(x^{\nu} , Ufs(x^{\nu}))\delta \beta) \in V^{0,0}_{\text{res.}, i, x}$ and $f^{0,0}_{i,m, V_{\text{res.}}} (x) \leq Ufs_i(x - C_{i}(x^{\nu} , Ufs(x^{\nu}))\delta \beta ) \leq f^{0,0}_{i,M, V_{\text{res.}}} (x ) $.}: $f^{0,1}_{i, m} (x ) \leq g_i(x) \leq f^{0,1}_{i, M}(x )$. Using the first relation of \eqref{eq:3.20} with $N = 0 $, $k = 1$, $\alpha = |\delta \beta|$ and noting that $\Delta f^{0,0} = 0$, we have
\begin{equation}
\label{eq:3.49}
\Delta f^{0,1} = C_2 (\delta \beta)^2
\end{equation}
since $Ufs_i (x)$ also lies within the upper and lower bounds for the solution $f^{0,1}_{i, m} (x) \leq Ufs_i(x ) \leq f^{0,1}_{i, M}(x )$, based on relation \ref{eq:3.49} we have $|g_i(x) - Ufs_i(x ) | = O (\delta \beta^2)$. Therefore
\begin{flalign}
\label{eq:3.50}
Ufs_i(x) - & Ufs_i(x - \hat{e}_m \delta \beta) = Ufs_i(x) - g_i(x) + g_i(x ) - Ufs_i(x - \hat{e}_m \delta \beta) \! = \nonumber \\ & O ({\delta \beta}^2) + Ufs_i(x - \hat{e}_m \delta \beta) \! - \! {\partial \over \partial x_l} Ufs_i(x - \hat{e}_m \delta \beta)C_{il}(x^{\nu},Ufs(x^{\nu})) \delta \beta + \nonumber \\ & \! R({\delta \beta}) + D_i(x^{\nu},Ufs(x^{\nu}))\delta \beta \! - \! Ufs_i(x - \hat{e}_m \delta \beta)
\end{flalign}
with $R(\delta \beta )/\delta \beta \rightarrow 0$ as $\delta \beta \rightarrow 0$ and we used the fact that $Ufs_i$ is differentiable at $x - \hat{e}_m \delta \beta$. Note that $x - \hat{e}_m \delta \beta \in V_{\beta}$ is a fixed point and $x \in V_{\beta + \delta \beta}$ is varied as $\delta \beta \rightarrow 0$. Another point to consider here is that we only used the fact that $Ufs_i$ is differentiable on $V_{\beta}$ and did not need to assume it is differentiable in the $x_m$ direction in \ref{eq:3.50}. Dividing relation \ref{eq:3.50} by $\delta \beta$ and taking the limit $\delta \beta \rightarrow 0$ we find
\begin{equation}
\label{eq:3.51}
C_{il}(x,Ufs(x)){\partial \over \partial x_l} Ufs_i(x) + {\partial \over \partial x_m }Ufs_i(x) = D_i(x,Ufs(x))
\end{equation}
This shows that $Ufs$ solves the PDE of relation \ref{theorem3.1} at its differentiable points subject to the initial condition \footnote[11]{Although the construction of $Ufs$ was done by moving in the positive $x_m$ direction it is clear that with similar methods it is possible to start from an initial condition hyperplane and construct the solution in the negative $x_m$ direction (c.f. \ref{appen:A}). Therefore the discussion here is equivalently valid for when making the replacement $\delta \beta \rightarrow - \delta \beta $ for $\delta \beta > 0$ and evaluating the derivative of $Ufs_i$ in the negative $x_m$ direction.}.

Next we will show that if the initial condition and the coefficients $C_{il}$ and $D_{i}$ are $C^1$ then $Ufs$ is $C^1$. We first show that $Ufs(x)$ is $C^1$ on $V^{N,k}$. We will make use of the following two theorems in mathematical analysis \cite{02}:
\begin{enumerate}
  \item Arzela-Ascoli theorem: Any bounded equicontinuous sequence of functions in $C^0(\prod^{d}_{h=1} [a_h,b_h], \mathbb{R})$ has a uniformly convergent subsequence.
  \item Theorem: The uniform limit of a sequence of functions in $C^1(\prod^{d}_{h=1} [a_h,b_h], \mathbb{R})$ is $C^1$ provided that the sequence of its partial derivatives also converges uniformly and the partial derivative of the uniform limit function is the same as the uniform limit of the partial derivative.
\end{enumerate}
Consider the collection of functions $f^{N,k}_i: V^{N,k} \rightarrow \mathbb{R}$ defined recursively as follows\begin{flalign}
\label{eq:3.52}
& f^{N,k}_i (x) = f^{N,k-1}_i ( x - C_i(x^{\nu} , f^{N,k-1}(x^{\nu})) \alpha/ 2^N ) + D_i(x^{\nu} , f^{N,k-1}(x^{\nu})) \alpha /2^N \nonumber \\
& f^{N,0}_i \equiv I_i , \ x \in V^{N,k} , \ x^{\nu}  = x - \nu {\alpha \over 2^N} , \ \nu = (m_{C_l} + M_{C_l}) \hat{e}_l /2 + \hat{e}_m
\end{flalign}
from the way the functions $f^{N,k}_i$ are defined it can be seen \footnote[12]{A similar reasoning as the footnote of the previous page holds here: $x^{\nu} \in S^{N,k}_{+, x}$, $(x^{\nu}, Ufs(x^{\nu})) \in P^{N,k}_{+, x}$ therefore $m^{N,k}_{D_i}(x ) \leq D_i(x^{\nu},Ufs(x^{\nu})) \leq  M^{N,k}_{D_i}(x)$ and $-M^{N,k}_{C_{il}}(x) \leq -C_{il}(x^{\nu},Ufs(x^{\nu})) \leq -m^{N,k}_{C_{il}}(x )$ hence $(x - C_{i}(x^{\nu} , Ufs(x^{\nu}))\alpha /2^N ) \in V^{N,k-1}_{\text{res.} , i, x}$ and $f^{N,k-1}_{i, m, V_{\text{res.}}} (x) \leq Ufs_i(x - C_{i}(x^{\nu} , Ufs(x^{\nu})) \alpha /2^N ) \leq f^{N,k-1}_{i,M, V_{\text{res.}}} (x ) $.}
\begin{equation}
\label{eq:3.53}
f^{N,k}_{i, m}(x) \leq f^{N,k}_i (x) \leq f^{N,k}_{i, M} (x) , \ \ \ x \in V^{N,k}
\end{equation}
we consider a fixed $V^{N,k_N}$ (for $1 \leq k_N \leq 2^N$) at $x_m = x_{0m} + q \alpha$ with $q = k_N / 2^{N}$ held fixed as $N \rightarrow \infty$. To show that $Ufs$ is differentiable on $V^{N,k_N}$ we have to show the following:
\begin{enumerate}
\item Uniform convergence of the sequence of functions $f^{N,k_N}_i$ on $V^{N,k_N}$.
\item Uniform convergence of the sequence (or at least a subsequence) of the partial derivatives ${\partial} f^{N,k_N}_i / \partial {x_l}$ on $V^{N,k_N}$.
\end{enumerate}
To show the uniform convergence of a subsequence of the partial derivatives it suffices to show the following:
\begin{enumerate}
\item [2.1] Boundedness of the sequence of partial derivatives ${\partial} f^{N,k_N}_i / \partial {x_l}$.
\item [2.2] Equicontinuity of the sequence of partial derivatives ${\partial} f^{N,k_N}_i / \partial {x_l}$.
\end{enumerate}

The first statement follows from relation \ref{eq:3.53} and the fact that as $N \rightarrow \infty$ the upper and lower bounds approach each other uniformly on $V^{N,k_N}$ as shown in \ref{eq:3.21}. To show statements $2.1$ and $2.2$ we take the partial derivative of \ref{eq:3.52}\footnote[13]{For brevity we have used the symbol $ H_{, x_l} \equiv {\partial H / \partial x_l} $.}
\begin{flalign}
\label{eq:3.54}
& {\partial \over \partial x_l} f^{N,\bar{k}}_i (x) = f^{N,\bar{k}-1}_{i , z_l}(z) + \left\{ - C_{ih ,x_l}(p^{\nu}) \alpha /2^N - C_{ih ,y_s}(p^{\nu}) f^{N,\bar{k} -1}_{s, x_l} (x^{\nu})  \alpha /2^N \right\} f^{N,\bar{k}-1}_{i , z_h}(z) \nonumber \\ & + D_{i,x_l}(p^{\nu}) {\alpha \over 2^N} + D_{i,y_s}(p^{\nu}) f^{N,\bar{k}-1}_{s, x_l}(x^{\nu}) {\alpha \over 2^N}, 1 \! \leq h \! \leq m - 1 , 1 \! \leq s \! \leq n, 1 \leq \! \bar{k} \! \leq k_N
\end{flalign}
with $p^{\nu} = (x^{\nu} , f^{N,k-1}(x^{\nu} ))$ and $z = x - C_i(p^{\nu}) \alpha/2^N$ in the above relation. Summation over $h$ and $s$ is implicit. To show the boundedness of the sequence of derivatives we assume a bound $ L^{N,\bar{k}-1}_f \geq |{\partial } f^{N,\bar{k}-1}_i(x') / \partial x'_l |$ is known for the partial derivatives of $f^{N,\bar{k}-1}_i(x')$ for $x' \in V^{N,\bar{k}-1}$ and look for $L^{N,\bar{k}}_f \geq |{\partial } f^{N,\bar{k}}_i(x) / \partial x_l |$. From \ref{eq:3.54} we can find such recursion relation
\begin{equation}
\label{eq:3.55}
 L^{N,\bar{k}}_f \! \equiv \! L^{N,\bar{k}-1}_f \!\! \left( \! 1 \! + \! (m\! - \! 1)L_C {\alpha \over 2^N} \! + nL_D {\alpha \over 2^N} \right) + n(m-1)L_C \! \left( L^{N,\bar{k}-1}_f \right)^2 \! \! \! {\alpha \over 2^N} + L_D {\alpha \over 2^N} \! \geq \! | f^{N,\bar{k}}_{i, x_l} (x) |
\end{equation}
where $L_C$ and $L_D$ are Lipschitz constants for $C_{il}(x,y)$ and $D_i(x,y)$ which bound $|C_{il,x_l}(x,y)|$, $|C_{i,y_s}(x,y)|$ and $|D_{i,x_l}(x,y)|$, $|D_{i,y_s}(x,y)|$ respectively, for $(x,y) \in P$ . Relation \ref{eq:3.55} is exactly similar to relation \ref{eq:3.34} obtained previously for the Lipschitz constants $L^{N,k}$. This proves $2.1$ that the sequence $f^{N,k_N}_{i, x_l}$ is bounded (locally in $\alpha$). To prove $2.2$ we have to show that the sequence $f^{N,k_N}_{i,x_l}$ is equicontinuous. The $C^1$ assumption for the initial condition and the coefficients $C_{il}(x,y)$ and $D_i(x,y)$ implies that $f^{N,k_N}_{i,x_l}$ is continuous and since they are defined on a compact set they are uniformly continuous, therefore we only have to show that for a $\epsilon > 0$ there is a common $\delta > 0$, independent of $N$, such that if $\Vert x - \widetilde{x} \Vert_1 < \delta \rightarrow | f^{N,k_N}_{i,x_l} (x) - f^{N,k_N}_{i,x_l} (\widetilde{x})| < \epsilon $, for $x, \widetilde{x} \in V^{N,k_N}$.

Taking the functions $f^{N,\bar{k} -1}_{i}(x')$ as known, for an $\epsilon^{N, \bar{k} -1} > 0$ choose $\delta^{N,\bar{k} -1} > 0$ such that if $\Vert x' - \widetilde{x}' \Vert_1 < \delta^{N,\bar{k}-1} $ for $x', \widetilde{x}' \in V^{N,\bar{k}-1}$ and $\Vert p' - \widetilde{p}' \Vert_1 < \delta^{N,\bar{k}-1} (1 + L_f) $ for $p', \widetilde{p}' \in P$ with $L_f$ given by \ref{eq:3.20}, then 
\begin{flalign}
\label{eq:3.55'}
& | f^{N,\bar{k}-1}_{i,x'_l} (x') - f^{N,\bar{k}-1}_{i,\widetilde{x}'_l} (\widetilde{x}')| < \epsilon^{N,\bar{k}-1} \nonumber \\ 
& | C_{il',y_s}(p') - C_{il',y_s}(\widetilde{p}') | < \epsilon^{N,\bar{k}-1} , \  | C_{il',x_l}(p') - C_{il',x_l}(\widetilde{p}') | < \epsilon^{N,\bar{k}-1}  \\ & | D_{i,y_s}(p') - D_{i,y_s}(\widetilde{p}') | < \epsilon^{N,\bar{k}-1} , \ | D_{i,x_l}(p') - D_{i,x_l}(\widetilde{p}') | < \epsilon^{N,\bar{k}-1} \nonumber
\end{flalign}
$l' =1, ..., m-1$. For these $\epsilon^{N,\bar{k}-1}$ and $\delta^{N,\bar{k}-1}$ lets see which $\epsilon^{N,\bar{k}}$ and $\delta^{N,\bar{k}}$ we will obtain for $f^{N,\bar{k}}_{i,x_l} $. For this lets evaluate $| f^{N,\bar{k}}_{i,x_l} (x) - f^{N,\bar{k}}_{i,x_l} (\widetilde{x})|$ using the right hand side of \ref{eq:3.54} for $x, \widetilde{x} \in V^{N,\bar{k}}$ and $\Vert x - \widetilde{x} \Vert_1 < \delta^{N,\bar{k}} $ . Note that the difference of the product of any number of terms can be written in terms of the difference of each of the terms multiplied by other terms, for example
\begin{flalign}
\label{eq:3.56}
A_1A_2 ... A_t - \widetilde{A}_1\widetilde{A}_2 ... \widetilde{A}_t = \delta A_1 A_2 ... A_t + \widetilde{A}_1 \delta A_2 A_3 ... A_t + ... + \widetilde{A}_1 \widetilde{A}_2 ... \widetilde{A}_{t-1} \delta A_t
\end{flalign}
for $\delta A_h \equiv A_h - \widetilde{A}_h$, $1 \leq h \leq t$. Therefore the difference of the right hand side of \ref{eq:3.54} can be written in terms of the difference of each of the terms at their corresponding two distinct points multiplied by other terms which are bounded. Their two distinct points are either $x^{\nu} = x - \nu {\alpha / 2^N}$ and $\widetilde{x}^{\nu} \equiv \widetilde{x} - \nu {\alpha / 2^N}$ or $p^{\nu} = (x^{\nu},f^{N,\bar{k} -1} (x^{\nu}))$ and $\widetilde{p}^{\nu} \equiv (\widetilde{x}^{\nu},f^{N,\bar{k} -1} (\widetilde{x}^{\nu}))$ or $z = x - C_i(p^{\nu})\alpha/2^N$ and $\widetilde{z} \equiv \widetilde{x} - C_i(\widetilde{p}^{\nu})\alpha/2^N$. A bound for the difference between these points are $\Vert p^{\nu} - \widetilde{p}^{\nu} \Vert_1 < \delta^{N,\bar{k}} (1 + L_f) $ or $\Vert x^{\nu} - \widetilde{x}^{\nu} \Vert_1 < \delta^{N,\bar{k}}$ or $ \Vert z - \widetilde{z} \Vert_1 < \delta^{N, \bar{k}} (1 + L_C(1 + L_f) \alpha /2^N )$. Assuming $\delta^{N, \bar{k}} (1 + L_C(1 + L_f) \alpha /2^N ) = \delta^{N,\bar{k}-1}$ (note that with this assumption $\delta^{N, \bar{k}} \leq \delta^{N, \bar{k}-1}$ and $\delta^{N, \bar{k}} (1 +L_f) \leq \delta^{N, \bar{k}-1} (1 +L_f)$) and using \ref{eq:3.55'} we can find a bound for $| f^{N,\bar{k}}_{i, x_l} (x) - f^{N,\bar{k}}_{i, x_l} (\widetilde{x}) |$
\begin{flalign}
\label{eq:3.58}
& | f^{N,\bar{k}}_{i, x_l} (x) - f^{N,\bar{k}}_{i, x_l} (\widetilde{x}) | < \epsilon^{N,\bar{k}-1} + \epsilon^{N,\bar{k}-1} G \alpha/2^N = \epsilon^{N,\bar{k}}
\end{flalign}
with $G \geq 0$ a bounded constant. Therefore the $\delta^{N, \bar{k}} $ $( \leq \delta^{N, \bar{k}-1} )$ and $\epsilon^{N,\bar{k}} $ $( \geq \epsilon^{N,\bar{k}-1} )$ obtained for $f^{N,\bar{k}}_{i, x_l}$ \footnote[14]{Note that for the $\delta^{N, \bar{k}}$ and $\epsilon^{N,\bar{k}}$ obtained, relation \ref{eq:3.55'} for the derivatives of $C_{il}$ and $D_{i}$ is also satisfied: $ p, \widetilde{p} \in P$, $\Vert p - \widetilde{p} \Vert_1 < \delta^{N, \bar{k}} (1 + L_f) \rightarrow | C_{il',y_s}(p) - C_{il',y_s}(\widetilde{p}) | < \epsilon^{N,\bar{k}} , \  | C_{il',x_l}(p) - C_{il',x_l}(\widetilde{p}) | < \epsilon^{N,\bar{k}}, \ | D_{i,y_s}(p) - D_{i,y_s}(\widetilde{p}) | < \epsilon^{N,\bar{k}} , \ | D_{i,x_l}(p) - D_{i,x_l}(\widetilde{p}) | < \epsilon^{N,\bar{k}}$ since $\delta^{N, \bar{k}} \leq \delta^{N, \bar{k}-1}$ and $\epsilon^{N,\bar{k}} \geq \epsilon^{N,\bar{k}-1}$. } in terms of $\delta^{N, \bar{k}-1}$ and $\epsilon^{N,\bar{k}-1}$ and eventually in terms of $\delta^{N, 0}$ and $\epsilon^{N,0}$ are as follows
\begin{flalign}
\label{eq:3.59}
& \epsilon^{N,\bar{k}} = \epsilon^{N,\bar{k}-1}(1 + G \alpha/2^N) = \epsilon^{N, 0}(1 + G \alpha/2^N)^{\bar{k}} \\ & \delta^{N, \bar{k}} = \delta^{N, \bar{k}-1} / (1 + L_C(1 + L_f) \alpha /2^N ) = \delta^{N, 0} / (1 + L_C(1 + L_f) \alpha /2^N )^{\bar{k}} \nonumber
\end{flalign}
for $\bar{k} = k_N = q 2^N$ we have
\begin{flalign}
\label{eq:3.60}
& \epsilon^{N, \ q 2^N} = \epsilon_0(1 + G \alpha/2^N)^{q 2^N} < \epsilon_{0} \exp (G q \alpha) = \epsilon \\ \label{eq:3.61} & \delta^{N, q2^N} = \delta_0 / (1 + L_C(1 + L_f) \alpha /2^N )^{q2^N} > \delta_0 / \exp ( L_C(1 + L_f) q \alpha) = \delta
\end{flalign}
where $\epsilon_0 = \epsilon^{N,0}$, $\delta_0 = \delta^{N,0}$. Therefore for a $\epsilon > 0$, we can choose $\epsilon_0$ small enough such that \ref{eq:3.60} is satisfied: $\epsilon_{0} \exp (G q \alpha) = \epsilon$. For this $\delta_0$ has to be chosen such that
\begin{flalign}
\label{eq:3.62}
& \Vert z - \widetilde{z} \Vert_1 < \delta_0 \rightarrow | I_{i,x_l} (z) - I_{i,x_l} (\widetilde{z})| < \epsilon_0 \ \ z , \widetilde{z} \in V \nonumber \\ & \Vert p - \widetilde{p} \Vert_1 < \delta_0 (1+ L_f) \rightarrow | C_{il',y_s}(p) - C_{il',y_s}(\widetilde{p}) | < \epsilon_0 , \\ & \hspace{0.5cm} | C_{il',x_l}(p) - C_{il',x_l}(\widetilde{p}) | < \epsilon_0 , | D_{i,y_s}(p) - D_{i,y_s}(\widetilde{p}) | < \epsilon_0 , | D_{i,x_l}(p) - D_{i,x_l}(\widetilde{p}) | < \epsilon_0 , \ \ p, \widetilde{p} \in P \nonumber
\end{flalign}
for the $\delta_0$ of \ref{eq:3.62} the $N$ independent $\delta$ is given by \ref{eq:3.61}: $\delta = \delta_0 / \exp ( L_C(1 + L_f) q \alpha)$. This shows that the sequence $f^{N,k_N}_{i,x_l}$ is equicontinuous and therefore statement 2.2 is proven. Therefore there exists a subsequence of $f^{N,k_N}_{i,x_l}$ for $l =1, ..., m-1$ that converges uniformly and since the sequence of $f^{N,k_N}_{i}$ converges uniformly to $Ufs_{i}$ on $V^{N,k_N}$ this shows that $Ufs_{i,x_l} (x)$ exists and is continuous in the direction of the variables $x_l$ for $l = 1, ..., m-1$ on $V^{N,k_N}$. Since the hyperplanes $V^{N,k_N}$ are dense in $S_+$ this easily generalizes to all hyperplanes parallel to the initial condition hyperplane in $S_+$ (e.g. by varying $\alpha$). Next we show that $Ufs_{i,x_l} (x)$ is continuous in the $x_m$ direction. Consider $V_{\beta}$ and $V_{\beta+\delta \beta}$ for $\delta \beta > 0$, defined at the beginning of Subsection \ref{sec:3.3}, as the initial condition and final hyperplane, respectively. We discretize the space in between along the $x_m$ direction similar to before. Consider \ref{eq:3.54} with $\alpha$ replaced by $\delta \beta$ and $V^{N,0}$ and $V^{N,2^N}$ corresponding to $V_{\beta}$ and $V_{\beta+ \delta \beta}$, respectively, with noting that all the terms have a bounded behaviour as $N \rightarrow \infty$ the recursion relation can be written as $f^{N,2^N}_{i,x_l} (x) = f^{N,2^N-1}_{i,x_l} (x - \hat{e}_m \delta \beta /2^N + \hat{e}_l O'_l( \delta \beta) /2^N ) + O'(\delta \beta) /2^N$ with $O'_l(\delta \beta)$ and $O'(\delta \beta)$ terms of order $\delta \beta$, therefore upon solving this relation for $f^{N,2^N}_{i,x_l} (x)$ ($x \in V_{\beta + \delta \beta}$) in terms of $f^{N,0}_{i,x'_l}(x' ) = Ufs_{i,x'_l}(x')$ ($x' \in V_{\beta}$), we find $f^{N,2^N}_{i,x_l}(x) = Ufs_{i,x_l}(x - \hat{e}_m \delta \beta + \hat{e}_l O^N_l( \delta \beta) ) + O^N(\delta \beta)$, with $x - \hat{e}_m \delta \beta + \hat{e}_l O^N_l( \delta \beta) \in V_{\beta}$, $O^N_l( \delta \beta)$ and $O^N( \delta \beta)$ terms of order $\delta \beta$. From $f^{N,2^N}_{i,x_l}$ there is a subsequence (e.g. $f^{a_n,2^{a_n}}_{i,x_l}$) that converges uniformly to $Ufs_{i ,x_l}(x)$, therefore \footnote[15]{$\lim_{n \rightarrow \infty} \hat{e}_l O^{a_n}_l(\delta \beta) \equiv \hat{e}_l O_l(\delta \beta)$ and $\lim_{n \rightarrow \infty} O^{a_n}(\delta \beta) \equiv O(\delta \beta)$, these limits are well defined. To see this consider $f^{a_n,2^{a_n}}_{i,x_l}(x) = Ufs_{i,x_l}(x - \hat{e}_m \delta \beta + \hat{e}_l O^{a_n}_l( \delta \beta) ) + O^{a_n}(\delta \beta)$, as noted $f^{a_n,2^{a_n}}_{i,x_l}(x)$ converges to $Ufs_{i, x_l}(x)$. $x - \hat{e}_m \delta \beta + \hat{e}_l O^{a_n}_l( \delta \beta) \in V_{\beta}$ converges to the point in $V_{\beta}$ which the characteristic curve of the solution $f_i$ passing through $x \in V_{\beta +\delta \beta}$ passes through in $V_{\beta}$, therefore the $O^{a_n}(\delta \beta)$ term also has a well defined limit as $n \rightarrow \infty$.}
\begin{flalign}
\label{eq:3.63}
Ufs_{i,x_l}(x) - & Ufs_{i,x_l}(x - \hat{e}_m \delta \beta ) = \lim_{n\rightarrow \infty}\{ f^{a_n, 2^{a_n}}_{i,x_l}(x ) \} - Ufs_{i,x_l}(x -\hat{e}_m \delta \beta) \nonumber \\ & = Ufs_{i,x_l}(x - \hat{e}_m \delta \beta + \hat{e}_l O_l( \delta \beta)) + O(\delta \beta) - Ufs_{i,x_l}(x - \hat{e}_m \delta \beta)
\end{flalign}
we already proved that $Ufs_{i,x_l}$ is continuous in the direction of the variables $x_l$ on $V_{\beta}$, therefore upon taking the limit $\delta \beta \rightarrow 0$ in \ref{eq:3.63} (note that for $x \in V_{\beta + \delta \beta}$, $x - \hat{e}_m \delta \beta \in V_{\beta}$ is a fixed point) it can be concluded that $Ufs_{i,x_l}$ is continuous in the $x_m$ direction\footnote[16]{As previously noted although the construction of $Ufs$ was done by moving in the positive $x_m$ direction it is clear that with similar methods it is possible to start from an initial condition hyperplane and construct the solution in the negative $x_m$ direction (c.f. \ref{appen:A}). Therefore the discussion in this page is equivalently valid for when making the replacement $\delta \beta \rightarrow - \delta \beta $ for $\delta \beta > 0$ and showing the continuity of $Ufs_{i,x_l}(x)$ in the negative $x_m$ direction.}. From \ref{eq:3.50} and \ref{eq:3.51} it follows that $Ufs_{i} (x)$ solves the system of PDE of \ref{eq:3.1} subject to the initial condition for all $x \in S_+$ and that $Ufs_{i, x_m} (x)$ exists and is continuous. Similarly with assuming that the initial condition and the coefficients $C_{il}$ and $D_i$ are $C^{r+1}$ for $r \geq 1$ we can show that the solution is $C^{r+1}$. For this consider the $r+1$ partial derivatives of \ref{eq:3.52}, by similar methods it can be shown that the sequence of a $r+1$ partial derivative of $f^{N,k_N}_{i}$ is bounded and equicontinuous and with a subsequence of its lower $r$ derivative converging uniformly, it can be concluded that the $r+1$ partial derivative of $Ufs_i$ in the $x_l$ directions exists and is continuous in the $x_l$ directions for $1 \leq l \leq m-1$, also similar to the argument above it can be concluded that the $r+1$ partial derivative in the $x_l$ directions is continuous in the $x_m$ direction. Then using \ref{eq:3.51} it can be shown that all $r+1$ partial derivatives in the $x_j$ direction for $j = 1, ..., m$ exist and are continuous.

Note that with the Lipschitz or $C^r$ assumption on the coefficients and the initial condition we obtain a Lipschitz or $C^r$ solution, respectively but the characteristic curves and the solution along these curves will be $C^1$ with Lipschitz continuous derivative and $C^{r+1}$, respectively as can be seen from relation \ref{eq:3.2}.

Although the solution was constructed on $S_+$ by a similar procedure we can define an $S_-$ domain and construct a unique solution there (c.f. \ref{appen:A}), it is also possible to extend the domain of the solution to a larger one by applying the same procedure on regions near the boundaries of the domain $S \equiv S_+ \cup S_-$. Further proceedings in the positive or negative $x_m$ direction, depending on the specific problem considered, might lead to regions of overlapping characteristics or an unbounded increase of the solution or its derivatives which would limit the domain with a well defined unique solution. Nevertheless we would expect there to exist a maximal domain with a unique well defined solution. For instance consider the union of all domains which a unique well defined solution exists with unique characteristics connecting the points of the domain to the initial condition domain. Other regions of the domain $P_1$ are regions which no solution, that is related to the initial condition, exists, i.e. there is no characteristic connecting that region to the initial condition domain, or multiple solutions exist with multiple characteristics connecting a point in that region to the initial condition domain. \qed
\end{proof}

\section{Generalizations and application of Theorem 3.1}
\label{sec:4}
\subsection{Dependence of initial condition and coefficients on parameters}
\label{sec:4.1}
In this Subsection we consider the dependence of the initial condition $I$ and coefficients $C_{il}$ and $D_i$ on parameters and show that their Lipschitz or $C^r$ dependence on the parameters is inherited to the solution. The Proposition is as follows:
\begin{prop} \onehalfspacing Consider extending the definition of $C_{il}$, $D_{i}$ and $I_i$ of Theorem \ref{theorem3.1} to $\bar{C}_{il}: P \times P_3 \rightarrow \mathbb{R}$, $\bar{D}_{i}: P \times P_3 \rightarrow \mathbb{R}$ and $\bar{I}_i: V \times P_3 \rightarrow \mathbb{R}$ with $P_3 \equiv \{ w \in \mathbb{R}^d |  \left\Vert w - w_0 \right\Vert_\infty \leq c \}$, $P \equiv P_1 \times P_2$ with $P_1$, $P_2$ and $V$ defined in Theorem \ref{theorem3.1}. Let $\bar{C}_{il}$, $\bar{D}_{i}$ and $\bar{I}_i$ be Lipschitz or $C^r$ with $\bar{I}_i (x , w_0) = I_i (x)$, $\bar{C}_{il}(x,y,w_0) = {C}_{il}(x,y)$ and $\bar{D}_{i}(x,y,w_0) = D_{i}(x,y)$, also let $M_{\Vert \bar{I} - y_0 \Vert} < b$ with $M_{\Vert \bar{I} - y_0 \Vert} \equiv \max \{ \Vert \bar{I}(u,w) - y_0 \Vert_{\infty} | (u,w) \in V \times P_3 \} $. Then the following system of partial differential equations:
\begin{equation}
\label{eq:4.1}
\bar{C}_{i1}(x,y,w) {\partial y_i \over \partial x_1} + ... + \bar{C}_{im-1}(x,y,w) {\partial y_i \over \partial x_{m-1}} + {\partial y_i \over \partial x_m} = \bar{D}_i(x,y,w)
\end{equation}
has a unique Lipschitz continuous or $C^r$ solution respectively, $\bar{f} : B \times P_3 \rightarrow P_2$ for $V \subset B \subseteq P_1$, $B$ containing a neighbourhood of $V_{int}$, with $V_{int}$ defined in Theorem \ref{theorem3.1} and $\bar{f}$ reducing to the initial condition function $\bar{I}$ on $V \times P_3$, $\bar{f}(u,w) = \bar{I}(u,w)$ for $(u, w) \in V \times P_3$.
\end{prop}

The construction of $Ufs$ which was done in Section \ref{sec:3} can similarly be done here for a fixed $w$ (or in other words for a spectator $w$ argument) by replacing the constants of the problem
\begin{equation}
\label{eq:4.2}
\begin{split}
& M_{\Vert D \Vert}, M_{\Vert C \Vert}, M_{C_l}, m_{C_l}, L_C, L_D \ \text{defined on} \  P \ \text{with} \ M_{\Vert \bar{D} \Vert}, M_{\Vert \bar{C} \Vert}, M_{\bar{C}_l}, m_{\bar{C}_l}, L_{\bar{C}}, L_{\bar{D}} \\ & \text{defined on} \ P \times P_3 \ \text{and} \ L_I \ \text{defined on} \ V \ \text{with} \ L_{\bar{I}} \ \text{defined on} \ V \times P_3
\end{split}
\end{equation}
and accordingly relation \ref{eq:3.20} and the relations in Section \ref{sec:3} that involve these constants would be modified in this way.

To show that $\bar{Ufs}(x,w)$ is Lipschitz with respect to its $w$ argument consider the sequence of functions in \ref{eq:3.52}. Now with the initial condition and coefficients depending on the parameter $w$ the recursion relation picks up a $w$ dependence
\begin{flalign}
\label{eq:4.3}
& \bar{f}^{N,k}_i (x, w) \! = \! \bar{f}^{N,k-1}_i \!\! \left( x - \bar{C}_i \left( x^{\bar{\nu}} , \bar{f}^{N,k-1} \left( x^{\bar{\nu}}, w \right), w \right) \! {\alpha \over 2^N} , w \right) \! + \! \bar{D}_i \left( x^{\bar{\nu}} , \bar{f}^{N,k-1}(x^{\bar{\nu}}, w),w \right) \! {\alpha \over 2^N} \nonumber \\
& \bar{f}^{N,0}_i (x,w) \equiv \bar{I}_i(x,w) , \ x \in \bar{V}^{N,k} , \ x^{\bar{\nu}}  = x - \bar{\nu} {\alpha \over 2^N} , \ \bar{\nu} = (m_{\bar{C}_l} + M_{\bar{C}_l}) \hat{e}_l /2 + \hat{e}_m
\end{flalign}

The sequence of $\bar{f}^{N,k_N}_i (x, w)$ converges uniformly to $\bar{Ufs}(x, w)$ on $\bar{V}^{N,k_N} \times P_3$ for $q = k_N/2^N$ fixed as $N \rightarrow \infty$ and $1 \leq k_N \leq 2^N$ as can be seen from relation \ref{eq:3.21} after applying \ref{eq:4.2}. Therefore if it is shown that $\bar{f}^{N,k_N}_i (x, w)$ has a bounded Lipschitz constant with respect to $w$, this implies that $\bar{Ufs}(x, w)$ is Lipschitz with respect to $w$ on all $\bar{V}^{N,k_N}$ which then easily generalizes to all points in the domain $\bar{S}^+$(e.g. by varying $\alpha$) . Lets assume $\bar{f}^{N,k-1}_i (x,w)$ is Lipschitz with Lipschitz constant $L_{\bar{f}}^{N,k-1}$ and try to find the Lipschitz constant of $\bar{f}^{N,k}_i (x,w)$. Consider two different points $(x,w), (\widetilde{x},\widetilde{w}) \in \bar{V}^{N,k} \times P_3$. We would like to find $L_{\bar{f}}^{N,k}$ such that $| \bar{f}^{N,k}_i (x,w) - \bar{f}^{N,k}_i (\widetilde{x},\widetilde{w}) | \leq L_{\bar{f}}^{N,k} \{ \sum_l |x_l - \widetilde{x}_l| + \sum^{d}_{u = 1} |w_u - \widetilde{w}_u| \}$. First lets evaluate the difference between each of the terms in \ref{eq:4.3}
\begin{flalign}
\label{eq:4.4}
& \left| \bar{D}_i \left( x^{\bar{\nu}} , \bar{f}^{N,k-1}(x^{\bar{\nu}}, w),w \right) - \bar{D}_i \left( \widetilde{x}^{\bar{\nu}} , \bar{f}^{N,k-1}(\widetilde{x}^{\bar{\nu}}, \widetilde{w}),\widetilde{w} \right) \right| \leq  L_{\bar{D}} \left( 1 + n L_{\bar{f}}^{N,k-1} \right) \\ & \hspace{9cm} \Big{\{} \sum_l |x_l - \widetilde{x}_l | + \sum_u |w_u - \widetilde{w}_u | \Big{\}} \nonumber \\ & \left| \bar{f}^{N,k-1}_i \!\! \left( x \! - \! \bar{C}_i \! \left( x^{\bar{\nu}} \! , \bar{f}^{N,k-1}(x^{\bar{\nu}} \!, w), w \right) \! {\alpha \over 2^N} , w \right) \! - \! \bar{f}^{N,k-1}_i \!\! \left( \widetilde{x} \! -\!  \bar{C}_i \! \left( \widetilde{x}^{\bar{\nu}} \! , \bar{f}^{N,k-1}(\widetilde{x}^{\bar{\nu}} \!, \widetilde{w}), \widetilde{w} \right) \!\! {\alpha \over 2^N} , \! \widetilde{w} \right) \right| \nonumber \\ & \label{eq:4.5} \leq L_{\bar{f}}^{N,k-1} \left\{ 1 + {\alpha \over 2^N} L_{\bar{C}} (m-1) \left( 1 + n L_{\bar{f}}^{N,k-1} \right) \right\} \Big{\{} \sum_l |x_l - \widetilde{x}_l | + \sum_u |w_u - \widetilde{w}_u | \Big{\}}
\end{flalign}
using the above relations we can find a bound for $\left| \bar{f}^{N,k}_i (x, w) - \bar{f}^{N,k}_i (\widetilde{x},\widetilde{w}) \right|$
\begin{flalign}
\label{eq:4.6}
& \left| \bar{f}^{N,k}_i (x, w) - \bar{f}^{N,k}_i (\widetilde{x},\widetilde{w}) \right| \leq \Big{\{} {\alpha \over 2^N } L_{\bar{D}} \left( 1 + n L_{\bar{f}}^{N,k-1} \right) + \nonumber \\  & L_{\bar{f}}^{N,k-1} \left( 1 + {\alpha \over 2^N} L_{\bar{C}} (m-1) \left( 1 + n L_{\bar{f}}^{N,k-1} \right) \right)  \Big{\}} \Big{\{} \sum_l |x_l - \widetilde{x}_l | + \sum_u |w_u - \widetilde{w}_u | \Big{\}}
\end{flalign}
from \ref{eq:4.6} we obtain a similar recursion relation as \ref{eq:3.34} (but with \ref{eq:4.2} applied) for the Lipschitz constants
\begin{equation}
\label{eq:4.7}
 L_{\bar{f}}^{N,k} \equiv L_{\bar{f}}^{N,k-1} \! \left( 1 + \big{(} (m-1)L_{\bar{C}} + n L_{\bar{D}} \big{)} {\alpha \over 2^N} \right) + n(m-1) L_{\bar{C}} {\alpha \over 2^N} \left( L_{\bar{f}}^{N,k-1} \right)^2 \! + L_{\bar{D}} {\alpha \over 2^N}
\end{equation}
This shows that the sequence of Lipschitz constants is locally bounded for all $N$ and $k$, therefore $\bar{Ufs}$ is also Lipschitz with respect to its parametric dependence with $L_{\bar{Ufs}} = \max \{ L_{\bar{f}}, M_{\Vert \bar{D} \Vert} + (m-1)  L_{\bar{f}} M_{\Vert \bar{C} \Vert} \}$ being its Lipschitz constant on $\bar{S}^+\times P_3$(or $\bar{S}^- \times P_3$). Next we show that $\bar{Ufs}$ is $C^1$ with respect to the $x$ and $w$ space. First we show this on the hyperplanes $\bar{V}^{N,k}$. For this it suffices to show Statements $1$, $2.1$ and $2.2$ in Subsection \ref{sec:3.3} for the sequence $\bar{f}^{N,k_N}_i (x, w)$ with $q = k_N / 2^N $ held fixed. Statement 1 was discussed below relation \ref{eq:4.3}: the uniform convergence of $\bar{f}^{N,k_N}_i (x, w)$ to $\bar{Ufs}(x, w)$ as $N \rightarrow \infty$ follows from relation \ref{eq:3.21} after applying \ref{eq:4.2}. $\bar{f}^{N,k_N}_{i,x_l} (x, w)$ and $\bar{f}^{N,k_N}_{i,w_u} (x, w)$ are locally bounded for all $N$ since the Lipschitz constant of $\bar{f}^{N,k_N}(x,w)$ obeys relation \ref{eq:4.7}, hence this shows Statement $2.1$. To show Statement $2.2$ take the partial derivative of \ref{eq:4.3} with respect to $w_u$ and $x_l$
\begin{equation}
\label{eq:4.8}
\begin{split}
& {\partial \over \partial w_u} \bar{f}^{N,\bar{k}}_i (x, \! w) \! = \! \bar{f}^{N,\bar{k}-1}_{i, w_u} \! (\bar{z}, \! w) \! + \! \left( \! - \bar{C}_{ih , w_u}(\bar{p}^{\bar{\nu}}) {\alpha \over 2^N} \! - \! \bar{C}_{ih ,y_s}(\bar{p}^{\bar{\nu}}) \bar{f}^{N,\bar{k} -1}_{s, w_u} (x^{\bar{\nu}} \! , \! w)  {\alpha \over 2^N} \! \right) \! \bar{f}^{N,\bar{k}-1}_{i , \bar{z}_h} \! (\bar{z}, w) \\ & + \bar{D}_{i,w_u}(\bar{p}^{\bar{\nu}}) {\alpha \over 2^N} + \bar{D}_{i,y_s}(\bar{p}^{\bar{\nu}}) \bar{f}^{N,\bar{k}-1}_{s, w_u}(x^{\bar{\nu}},w) {\alpha \over 2^N}, \\
& {\partial \over \partial x_l} \bar{f}^{N,\bar{k}}_i (x, w) \! = \! \bar{f}^{N,\bar{k}-1}_{i , \bar{z}_l}( \bar{z}, w) + \left( \! - \bar{C}_{ih ,x_l}(\bar{p}^{\bar{\nu}}) {\alpha \over 2^N} - \bar{C}_{ih ,y_s}(\bar{p}^{\bar{\nu}}) \bar{f}^{N,\bar{k} -1}_{s, x_l} (x^{\bar{\nu}})  {\alpha \over 2^N} \right) \! \bar{f}^{N,\bar{k}-1}_{i , \bar{z}_h}( \bar{z}, w) \\ & + \bar{D}_{i,x_l}(\bar{p}^{\bar{\nu}}) {\alpha \over 2^N} + \bar{D}_{i,y_s}(\bar{p}^{\bar{\nu}}) \bar{f}^{N,\bar{k}-1}_{s, x_l}(x^{\bar{\nu}}, w) {\alpha \over 2^N} \! , \\ & \bar{p}^{\bar{\nu}} \! \equiv \!\! \left( x^{\bar{\nu}} , \bar{f}^{N,k-1}(x^{\bar{\nu}} , w), w \right), \ \bar{z} = x - \bar{C}_i \left( \bar{p}^{\bar{\nu}} \right) {\alpha \over 2^N}, 1 \leq h \leq m-1 , 1 \leq s \leq n, 1 \leq \bar{k} \leq k_N
\end{split}
\end{equation}
summation over $h$ and $s$ is implicit. Showing that the sequences $\bar{f}^{N,k_N}_{i,w_u} (x, w)$ and $\bar{f}^{N,k_N}_{i,x_l} (x, w)$ are equicontinuous is similar to how this was done for the partial $x_l$ derivatives in Subsection \ref{sec:3.3} as the structure of the recursion relation is the same, therefore by similar arguments starting from the paragraph below relation \ref{eq:3.62} until a few sentences after relation \ref{eq:3.63} we can conclude that $\partial \bar{Ufs}_i / \partial w_u$ and $\partial \bar{Ufs}_i / \partial x_j$ ($j =1, ..., m$) exist and are continuous with respect to the $x$ and $w$ space. Also with similar arguments as in the paragraph below relation \ref{eq:3.63} we can conclude that with a $C^{r+1}$ assumption on $\bar{C}_{il}$, $\bar{D}_i$ and $\bar{I}_i$, $\bar{Ufs}(x,w)$ will be $C^{r+1}$ with respect to $x$ and $w$.

\subsection{Generalization to nonlinear systems of PDE}
\label{sec:4.2}
In this Subsection we will generalize the result of Section \ref{sec:3} to nonlinear systems of PDE. For this we need to conjecture the following for a linear homogeneous first order system of PDE that will be derived later in this Subsection.
\begin{conj1}
\onehalfspacing
\label{conj1}
The following linear homogeneous first order system of PDE:
\begin{equation}
\label{eq:4.9'}
{\partial y \over \partial x_m} + A_l (x) {\partial y \over \partial x_l} + B(x)y = 0
\end{equation}
with $A_l (x)$ and $B (x) $, $n \times n$ $C^1$ matrices defined on $P_1$, can have at most one $C^1$ solution locally that satisfies a $C^1$ initial condition $I_i: V \rightarrow \mathbb{R}$, $y_i(u) = I_i(u)$ , $u \in V$, with $P_1$ and $V$ defined similar to Theorem \ref{theorem3.1}.
\end{conj1}
\textbf{Note}: If the matrices $A_l$ in \ref{eq:4.9'} are symmetric the above conjecture is true according to \cite{03}.

The nonlinear system of PDE that is reducible to the system of PDE of Theorem \ref{theorem3.1} by differentiation is \footnote[17]{The derivation presented here is similar to the one in \cite{01} except that it is for a system of PDE.}:
\begin{equation}
\label{eq:4.10}
G_i (x, y, \nabla y_i) = 0 , \ \ \  i = 1, ..., n
\end{equation}
we assume $G_i: P_1 \times P_2 \times Q_i \rightarrow \mathbb{R}$ is defined with $Q_i \equiv \{ z \in R^m | \Vert z - p_{i0} \Vert_{\infty} \leq c \}$, the points $p_{i0} \in \mathbb{R}^m$ and $c > 0$ will be defined below. The initial condition is given by $I_i: V \rightarrow \mathbb{R}$ and we demand that the functions $y_i$ solving \ref{eq:4.10} reduce to $y_i (u) = I_i (u)$ for $u \in V$ and $M_{\Vert I - y_0 \Vert} < b$. $P_1$, $P_2$, $V$ and $M_{\Vert I - y_0 \Vert}$ are defined similar to Theorem \ref{theorem3.1}. With the $C^{3}$ assumption on $G_i$ and the initial condition $I_i$ we will obtain a $C^{3}$ solution. In order for the existence of a solution to \ref{eq:4.10} that reduces to the initial condition on $V$ to be possible the functions $p^0_{i} = p^0_{i}(u)$ for $ u \in V$ must exist which satisfy the following relations:
\begin{flalign}
\label{eq:4.11}
& G_i (u, I(u), p^0_{i}(u)) = 0 , \\ \label{eq:4.12}
& p^0_{il}(u) - { \partial I_i(u) \over \partial u_l} = 0, \\ \label{eq:4.13}
& { {\partial G_{i} \over \partial p_{im}} } \ne 0 \ \ \ \ \  \text{at}  \ \ \ (u, I(u), p^0_i(u)) 
\end{flalign}
$p^0_{il}$ is $C^{2}$ from \ref{eq:4.12}, therefore due to the implicit function theorem $p^0_{im}$ will also be $C^{2}$ \footnote[18]{It is usually assumed that relations \ref{eq:4.11} - \ref{eq:4.13} hold for a point $u_0 \in V$ which then due to the implicit function theorem it can be inferred that they hold locally in a neighbourhood of $u_0 \in V$. Since we want the solution to reduce to the initial condition on $V$ we have assumed that \ref{eq:4.11} - \ref{eq:4.13} hold on $V$.}. $p_{i0}$ and $c > 0$ are such that $\forall u \in V \rightarrow \Vert p^0_{i}(u)  - p_{i0} \Vert_{\infty} < c$. Differentiating \ref{eq:4.10} with respect to $x_j$ we have
\begin{equation}
\label{eq:4.14}
\begin{split}
& G_{i, p_{iq}} (x, y, p_i) { \partial p_{ij} \over \partial x_q } = - {\partial G_i \over \partial x_j} - {\partial G_i \over \partial y_s} p_{sj} \ \ , \ \ \  s = 1, ..., n , \ \ \ q = 1, ..., m \\ & G_{i, p_{iq}} (x, y, p_i) { \partial y_{i} \over \partial x_q } = G_{i, p_{iq}} (x, y, p_i) p_{iq} 
\end{split}
\end{equation}
summation over $q$ and $s$ is implicit. We have commuted the order of the partial derivatives and replaced ${\partial y_i / \partial x_j} \rightarrow p_{ij}$. From \ref{eq:4.13} it is clear that ${ {\partial G_{i} / \partial p_{im}} } \ne 0 $ in a neighbourhood of the set of points $(u, I(u), p^0_i(u)) \in P_1 \times P_2 \times Q_i$ for $u \in V$, therefore upon dividing \ref{eq:4.14} by $G_{i, p_{im}} (x, y, p_i)$ we obtain a system of PDE similar to relation \ref{eq:3.1} which then a unique solution $y(x)$ and $p_i(x)$ can be constructed locally that would reduce to $y(u) = I(u)$ and $p_i(u) = p^0_i(u)$ for $u \in V$ similar to the way it was done in Section \ref{sec:3}. With the $C^{3}$ assumption on the initial condition $I_i$ and $G_i$ in \ref{eq:4.10}, the coefficients and the initial condition in \ref{eq:4.14} will be $C^{2}$ and therefore we obtain a $C^{2}$ solution to \ref{eq:4.14}.

Now it is possible to show that the $y_i$ of \ref{eq:4.14} solve the system of PDE of \ref{eq:4.10} and are $C^{3}$ assuming Conjecture \ref{conj1} holds. For this we introduce new coordinate systems corresponding to the initial condition hyperplane and the parameter of the characteristic equations of \ref{eq:4.14}. We denote these by $u^{(i)}_1, ...., u^{(i)}_{m-1}$ and $(u^{(i)}_{m} \equiv)$ $t^{(i)}$. By the theory of ordinary differential equations the map $x = x(u^{(i)})$ is $C^{2}$. To show that $y$ solves \ref{eq:4.10} we need to show the following equations:
\begin{flalign}
\label{eq:4.15}
G_i (x, y (x), p_i(x)) = 0 , \\ \label{eq:4.16}
{\partial y_i (x) \over \partial x_j} - p_{ij} (x) = 0
\end{flalign}
from \ref{eq:4.11} it is clear that \ref{eq:4.15} and from \ref{eq:4.12} and the second equation of \ref{eq:4.14} it is clear that \ref{eq:4.16} are true on the initial condition hyperplane $V$. We need to show that they hold locally near the initial condition hyperplane. Showing \ref{eq:4.16} is equivalent to showing $d y_i = p_{ij} dx_j$. In the coordinate system of the initial condition hyperplane $u^{(i)}_1 , ..., u^{(i)}_{m-1}$ and the characteristic parameter $( u^{(i)}_{m} \equiv ) $  $t^{(i)}$ this is equivalent to the following relations:
\begin{flalign}
\label{eq:4.17}
& \lambda_{i l} (u^{(i)}) \equiv {\partial y_i \over \partial u^{(i)}_l} - p_{iq} {\partial x_q \over \partial u^{(i)}_l} = 0 \\ & \label{eq:4.18}
\lambda_{i m} (u^{(i)}) \equiv {\partial y_i \over \partial t^{(i)}} - p_{iq} {\partial x_q \over \partial t^{(i)}} = 0
\end{flalign}
\ref{eq:4.18} is automatically satisfied from the second equation of \ref{eq:4.14} and the characteristic relation $\partial x_q / \partial t^{(i)} = G_{i,p_{iq}}$. \ref{eq:4.17} and \ref{eq:4.15} need to be shown. For this we take a derivative with respect to $t^{(i)}$ of these equations. We have
\begin{flalign}
\label{eq:4.19}
 {\partial \over \partial t^{(i)} }G_i (x, y(x), p_i(x)) & = G_{i, x_j} {\partial x_j \over \partial t^{(i)} } + G_{i,y_s} {\partial y_s \over \partial x_j} { \partial x_j \over \partial t^{(i)}} + G_{i, p_{ij}} {\partial p_{ij} \over \partial t^{(i)}} \nonumber \\ & = G_{i, x_j} G_{i, p_{ij}} + G_{i,y_s} G_{i, p_{ij}} {\partial y_s \over \partial x_j}  + G_{i, p_{ij}} \left( - {\partial G_i \over \partial x_j} - {\partial G_i \over \partial y_s} p_{sj} \right) \nonumber \\ & = G_{i,y_s} G_{i, p_{ij}} \left( {\partial y_s \over \partial x_j} - p_{sj} \right) = G_{i,y_s} G_{i, p_{ij}} {\partial u^{(s)}_q \over \partial x_j } \lambda_{s q} (u^{(s)})
\end{flalign}
\begin{flalign}
\label{eq:4.20}
{\partial \over \partial t^{(i)} } \lambda_{il} (u^{(i)}) & = {\partial \over \partial u^{(i)}_l} \{ G_{i, p_{iq}} p_{iq} \} - \left( - {\partial G_i \over \partial x_q} - {\partial G_i \over \partial y_s} p_{sq} \right) {\partial x_q \over \partial u^{(i)}_l} - p_{iq} {\partial G_{i, p_{iq}} \over \partial u^{(i)}_l} \nonumber \\ & = G_{i, p_{iq}} {\partial p_{iq} \over \partial u^{(i)}_l} + {\partial G_i \over \partial x_q} {\partial x_q \over \partial u^{(i)}_l} + {\partial G_i \over \partial y_s} p_{sq}{\partial x_q \over \partial u^{(i)}_l} = {\partial G_i \over \partial u^{(i)}_l} - {\partial G_i \over \partial y_s} {\partial y_s \over \partial u^{(i)}_l} + {\partial G_i \over \partial y_s} p_{sq}{\partial x_q \over \partial u^{(i)}_l} \nonumber \\ & = {\partial G_i \over \partial u^{(i)}_l} - {\partial G_i \over \partial y_s} {\partial u^{(s)}_j \over \partial u^{(i)}_l  } \lambda_{sj}(u^{(s)})
\end{flalign}
summation over $s, q $ and $j$ is implicit. The change of the order of the partial derivatives are allowed since $ y_i $ and $x_j $ are $C^{2}$ (If we had started with a $C^{2}$ assumption on the $G_i$ and the initial condition functions $I_i$, $y_i$ and the $x_q$ of \ref{eq:4.14} would have been $C^1$ as a function of $u^{(i)}_j$ but the change of the order of the partial derivatives in \ref{eq:4.20} would still be allowed since $\partial y_i / \partial t^{(i)}$ and $\partial x_q / \partial t^{(i)}$ are $C^1$). We also used the fact that the inverse map $u^{(i)} = u^{(i)} (x) $ is differentiable, in particular $C^{2}$, in the above relations, this follows from the fact that the map $x = x(u^{(i)})$ is $C^{2}$ and we know that $\det \{ {\partial x_j / \partial u^{(i)}_q} \} \ne 0$ near the initial condition hyperplane since at the initial condition hyperplane the coordinates $u^{(i)}_1, ...., u^{(i)}_{m-1}$ are the same as $x_1, ...., x_{m-1}$ and $dx_{m} = G_{i,p_{im}}  dt^{(i)}$ for $G_{i,p_{im}} \ne 0 $. Considering everything as a function of $x$ with $u^{(i)} = u^{(i)}(x)$ and rewriting $\partial / \partial t^{(i)}$ in terms of partial derivatives with respect to $x$ and noting that $\lambda_{im} = 0$, we obtain
\begin{equation}
\label{eq:4.21}
\begin{split}
& G_{i, p_{iq}} { \partial \bar{G}_i (x) \over \partial x_q } = G_{i,y_s} G_{i, p_{ij}} {\partial u^{(s)}_{l'} \over \partial x_j } \bar{\lambda}_{s l'} (x) \\
& G_{i, p_{iq}} { \partial \bar{\lambda}_{il} (x) \over \partial x_q } - {\partial \bar{G}_i (x) \over \partial x_q} {\partial x_q \over \partial u^{(i)}_l} = - G_{i, y_s} {\partial u^{(s)}_{l'} \over \partial u^{(i)}_l  } \bar{\lambda}_{sl'}(x)
\end{split}
\end{equation}
$l' \! = \! 1, .., m- 1$. Considering the coefficients of ${ \partial \bar{G}_i / \partial x_q }$, $ \partial \bar{\lambda}_{il} / \partial x_q$ and $\bar{\lambda}_{sl'}$ as known functions of $x$, which based on the assumptions of the theorem are $C^{1}$, it can be seen that the functions $\bar{G}_i(x) \equiv G_i (x, y(x), p_i(x)) $ and $\bar{\lambda}_{il}(x) \equiv \lambda_{il} (u^{(i)}(x))$ satisfy a linear homogeneous first order partial differential equation\footnote[19]{In order for $\lambda_{sl}$ to satisfy a linear system of PDE, $y_s$ should be at least two times differentiable, this is the main reason the coefficients and the initial condition in \ref{eq:4.10} were assumed $C^{3}$ so that the solution of \ref{eq:4.14} and in particular $y_s$ would be $C^{2}$. We do not rule out the possibility of improving this $C^{3}$ differentiability assumption. For example  with a $C^{2}$ assumption on the coefficients $G_i$ and $I_i$, \ref{eq:4.19} and \ref{eq:4.20} are still valid as the change of the order of partial derivatives is still allowed as mentioned in the sentences below equation \ref{eq:4.20}. In this case \ref{eq:4.19} and \ref{eq:4.20} are linear homogeneous partial differential equations for $G_i$ and $\lambda_{sl}$ in different coordinate systems(!) with coefficients that are at least continuous and it obviously has a solution of zero based on an initial condition of zero. If this can be defined properly and a similar conjecture as Conjecture 1 holds for it then it is possible to start with a $C^{2}$ assumption on $G_i$ and $I_i$. A more optimum differentiability assumption is that we start with a $C^1$ assumption with Lipschitz continuous derivatives on $G_i$ and $I_i$, in this case if we assume there is a $C^1$ solution with Lipschitz continuous derivatives to \ref{eq:4.10} then this solution will inevitably be given by the unique Lipschitz solution to \ref{eq:4.14}. In this case it might be possible to show that this Lipschitz solution solves \ref{eq:4.10} near $V$ as this is true for when we only have one equation with one unknown function (when $n =1$) in \ref{eq:4.10} as stated in \cite{01}.} in the form of relation \ref{eq:4.9'} \footnote[20]{With $G_{i,p_{im}} \ne 0$ near $V$ after dividing \ref{eq:4.21} by $G_{i,p_{im}}$ we obtain a similar form as \ref{eq:4.9'}. Note that the term $\sim \partial \bar{G}_i(x) / \partial x_m$ in the second relation of \ref{eq:4.21} can be eliminated by multiplying the first relation of \ref{eq:4.21} by $ (\partial x_m / \partial u^{(i)}_l ) / G_{i,p_{im}}$ and adding it to the second relation of \ref{eq:4.21}.}, and since they vanish on the initial condition hyperplane $V$, assuming Conjecture \ref{conj1} holds, they should also vanish near the initial condition hyperplane \footnote[21]{\ref{eq:4.21} clearly has a solution of zero based on an initial condition of zero. We might ask the question as to whether this is a unique solution.  Here we will try to argue in favour of a unique solution. Having another solution other than zero would lead to some unsatisfactory results. For example if we have a non-zero solution then a constant multiple of that solution would also be a solution based on an initial condition of zero and this generates an infinite family of solutions. Or considering the discretization of the system of PDE of \ref{eq:4.21} the values of the discretized solution obtained at each discretized hyperplane parallel to the initial condition hyperplane would all be zero, therefore it seems that a non-trivial solution cannot be captured by the discretization of the system of PDE of \ref{eq:4.21}. Also from \cite{03} it is known that when the $A_l$ matrices are symmetric, \ref{eq:4.9'} can have at most one solution. Therefore it seems plausible to conjecture that Conjecture \ref{conj1} holds for general $n \times n$ $C^1$ matrices $A_l$ and accordingly the linear homogeneous first order system of PDE of \ref{eq:4.21} would admit at most one solution locally based on an initial condition of zero.}. Therefore \ref{eq:4.15} and relations \ref{eq:4.17} and \ref{eq:4.18} (or equivalently \ref{eq:4.16}) are valid near $V$. This shows that $y = y(x)$ of \ref{eq:4.14} solves the system of PDE of \ref{eq:4.10} near $V$ and since $p_{ij}$ is $C^{2}$, $y$ would be $C^{3}$. It is also possible to combine the results of Subsections \ref{sec:4.1} and \ref{sec:4.2} easily by extending the definition of the initial condition and $G_i$ functions of \ref{eq:4.10} to have a parametric dependence $w$ on a compact parameter space. With assuming the conditions and assumptions in this Subsection hold for any fixed $w$ (or in other words for any spectator $w$ argument) in the compact parameter space the discussion of this Subsection is similarly valid without any change. The only point to note is that with a $C^{3}$ assumption on the initial condition and the functions $G_i$, the solutions obtained for the system of PDE of \ref{eq:4.14} will be $C^{2}$ with respect to the $x$ and $w$ space, therefore $y_i$ and ${\partial y_i / \partial x_j }$ ($ = p_{ij}$) will be $C^{2}$ with respect to the $x$ and $w$ space.

At the end of this Subsection we note that the initial condition can also be defined on an arbitrary hypersurface instead of a hyperplane. In this case it is possible to reduce the problem to one that is defined on a hyperplane by a change of variables. Consider the following $C^{3}$ hypersurface $x: U \rightarrow \mathbb{R}^m$, $x = x (u_1, ..., u_{m-1})$ for $U \equiv [-a,a]^{m-1}$, and $\partial x (u) / \partial u $ has rank $m-1$. We demand that the functions $y$ solving \ref{eq:4.10} reduce to $y(x(u)) = I(u)$ on this hypersurface for some set of $C^{3}$ initial condition functions $I_i: U \rightarrow \mathbb{R}$. Since the rank of ${\partial x / \partial u}$ is $m-1$ at any point $u_0 \in U$ there exists $m-1$ rows of the matrix ${\partial x / \partial u}$ that are linearly independent. Without loss of generality we take the first $m-1$ rows to be linearly independent. Therefore we can change coordinates from $u_1, ..., u_{m-1}$ to $x_1, ..., x_{m-1}$ near $u = u_0$. Since $\det \{ {\partial x_l / \partial u_{l'}} \} \ne 0$ ($l, l' = 1, ..., m-1$) near $u_0$ the inverse map $u = u(x_1, ..., x_{m-1})$ is also $C^{3}$. Next we change coordinates from $( x_1, ..., x_{m-1}, x_m (u(x_1, ..., x_{m-1})) )$ to $( x_1, ..., x_{m-1}, x'_m )$ with $x'_m = x_m - x_m(u (x_1, ..., x_{m-1}))$ and in this new coordinate system the hypersurface near $x(u_0)$ is given by $x'_m = 0$. Note that the functions $G_i$ of \ref{eq:4.10} and the initial condition functions $I_i$ remain $C^{3}$ in this new coordinate system, $G_i (x, y ,\nabla y_i) = G_i (x_1, ..., x_{m-1}, x'_m + x_m (u(x_1, ..., x_{m-1})) , y , {\partial y_i / \partial x_1}, ..., {\partial y_i / \partial x_{m-1}}, {\partial y_i / \partial x'_{m}} )$, $I_i = I_i(u(x_1, ..., x_{m-1}))$.

\subsection{Application to hyperbolic quasilinear systems of first order PDE in two independent variables}
\label{sec:4.3}
In this Subsection we will show that a hyperbolic quasilinear system of first order PDE in two independent variables can be reduced to the system of PDE of Theorem \ref{theorem3.1}. Consider the following hyperbolic quasilinear system of first order PDE in two independent variables $x_1$ and $x_2$
\begin{flalign}
\label{eq:4.22}
{\partial y \over \partial x_2} + A(x, y) {\partial y \over \partial x_1} = B(x, y)
\end{flalign}
$A(x,y)$ and $B(x, y)$ are $n \times n$ and $n \times 1$ $C^{1}$ matrices, respectively, with Lipschitz continuous derivatives defined on $P_1 \times P_2$ with $P_1$, for $m=2$, and $P_2$ defined similar to Theorem \ref{theorem3.1}. It is assumed that $A$ has $n$ real eigenvalues $\tau^i(x, y)$ which form a diagonal matrix $\mathrm{T}(x,y)$ and $n$ linearly independent left eigenvectors $l^i(x, y)$ which form a matrix $\Lambda (x,y)$ with determinant one, $\mathrm{T}$ and $\Lambda$ are also considered $C^1$ with Lipschitz continuous derivatives \footnote[22]{For when the eigenvalues $\tau^i$ are distinct this follows from the fact that $A$ is $C^1$ with Lipschitz continuous derivatives.}. Furthermore we demand that the functions $y_i$ reduce to a set of initial condition functions on $V$, $f_i(u) = I_i(u)$ for $u \in V$, $I_i: V \rightarrow \mathbb{R}$ being $C^1$ with Lipschitz continuous derivatives and $V$, for $m=2$, defined similar to Theorem \ref{theorem3.1}.

To reduce the system of PDE above to the form of Theorem \ref{theorem3.1}  take the derivative of \ref{eq:4.22} with respect to $x_r$
\begin{flalign}
\label{eq:4.23}
{\partial p_{r} \over \partial x_2} + A {\partial p_{r} \over \partial x_1} = B_{,x_r} + p_{sr} B_{, y_{s}} - A_{,x_r} p_{1} - p_{sr} A_{, y_{s}} p_{1} \equiv C(x,y,p_1, p_2)
\end{flalign}
summation over $s =1, .., n$ is implicit, we have changed the order of the partial derivatives \footnote[23]{The change of the order of derivatives is allowed almost everywhere since based on the differentiability assumptions on $A$, $B$ and $I_i$, the partial derivative of the solution, $\partial y / \partial x_r$, will be Lipschitz and therefore is differentiable almost everywhere.} and replaced $\partial y / \partial x_r \rightarrow p_{r}$ for $r=1,2$ and $p_{sr} \equiv ( p_{r} )_{s}$. Next multiply \ref{eq:4.23} by $\Lambda$ and define the new function variables $\bar{p}_r \equiv \Lambda p_{r}$, we have 
\begin{flalign}
\label{eq:4.24}
{\partial \bar{p}_r \over \partial x_2} + \mathrm{T} {\partial \bar{p}_r \over \partial x_1} = (\Lambda_{,x_2} + \Lambda_{,y_s}p_{s2} ) p_r  + \mathrm{T} ( \Lambda_{,x_1} + \Lambda_{,y_s} p_{s1}) p_r + \Lambda C
\end{flalign}
and the PDE for $y$ is given by
\begin{flalign}
\label{eq:4.25}
{\partial y \over \partial x_2} + \mathrm{T} {\partial y \over \partial x_1} = p_2 + \mathrm{T} p_1
\end{flalign}
(or $\partial y / \partial x_2 = p_2$ is also a valid choice instead of \ref{eq:4.25}) the system of PDE of \ref{eq:4.24} and \ref{eq:4.25}, in terms of the functions $y$ and $\bar{p_r}$ (with $p_r = \Lambda^{-1} \bar{p}_r$), has the form of Theorem \ref{theorem3.1} with coefficients and initial condition that are Lipschitz. The initial condition is given by $y(u) = I(u)$, $\bar{p}_1(u) = \Lambda (u, I(u)) \partial I (u) / \partial u_1 $ and $\bar{p}_2(u) = \Lambda (u, I(u)) \{ B(u,I(u)) - A(u, I(u))  \partial I (u) / \partial u_1 \} $ for $u \in V$. From \cite{03} it is known that \ref{eq:4.22} has a local unique $C^1$ solution with Lipschitz continuous derivatives that satisfies the initial condition, therefore it is clear that this solution is given by the local unique Lipschitz solution of \ref{eq:4.24} and \ref{eq:4.25}: $y$ and $\partial y / \partial x_r = p_r = \Lambda^{-1} \bar{p}_r$. This shows that Theorem \ref{theorem3.1} gives an alternative way, which is more direct and convenient especially for finding a numerical solution (e.g. The discretized form of the solution can be obtained by considering relation \ref{eq:3.52} for the system of PDE of \ref{eq:4.24} and \ref{eq:4.25}), as compared to other methods, e.g. iteration methods \cite{03}, for the construction of the solution of hyperbolic quasilinear systems of first order PDE in two independent variables.

\appendix
\section{}
\label{appen:A}
\onehalfspacing
In this Appendix we will list the equivalent definitions and relations of Section \ref{sec:3} for when constructing a solution on the $S_-$ domain. The $S_-$ domain is defined as
\begin{equation}
\label{eq:A1}
S_- \! \equiv \! \left\{ x \in P_1 \! \left| -\alpha \leq x_m - x_{0 m} \leq 0 , - \bar{a} + m_{C_l} ( x_m - x_{0 m} ) \leq  x_l - x_{0 l}  \leq \bar{a} + M_{C_l} ( x_m - x_{0 m} ) \right. \right\}
\end{equation}
and $\alpha$ satisfies the 3 conditions listed below relation \ref{eq:3.3}. $M_{C_l}$ and $m_{C_l}$, similar to before, refer to an upper and lower bound for $C_{il}$ for $i =1, ..., n$ on P, respectively. Relation \ref{eq:3.4} is modified to
\begin{flalign}
\label{eq:A2}
& f^{N, k}_{i,m}(x) \equiv f^{N, k-1}_{i,m,V_{\text{res.}}}(x) - M^{N,k}_{ D_i}(x) {\alpha \over 2^N} \leq f_i(x) \leq f^{N, k-1}_{i,M,V_{\text{res.}}}(x) - m^{N,k}_{D_i}(x) {\alpha \over 2^N} \equiv f^{N, k}_{i, M}(x) \nonumber \\ & x \in V^{N,k}_-, V^{N,k}_- \equiv \left\{ z \in S_- \left| z_m = x_{0m} - k \alpha / 2^N \right. \right\}, \ V^{N,0}_- \equiv V, \ k =1, ...,2^N  &&
\end{flalign}
$f^{N, 0}_{i,M}(x) = f^{N, 0}_{i,m}(x) \equiv I_i(x)$ for $x \in V$. $f^{N, k}_{i, M}, f^{N, k}_{i, m}: V^{N,k}_- \rightarrow \mathbb{R} $. $f^{N, k-1}_{i, M ,V}(x)$, $f^{N, k-1}_{i, m ,V}(x)$, $M^{N,k}_{D_i}(x)$ and $m^{N,k}_{D_i}(x)$ for $x \in V^{N,k}_-$ are given by
\begin{equation}
\label{eq:A3}
\begin{split}
 & \hspace{-0.3cm} f^{N, k-1}_{i, M ,V}(x) \equiv \max \left\{ f^{N, k-1}_{i, M}(z) \left| z \in S^{N,k}_{-, x} \cap V^{N,k-1}_- \right. \right\},\\ & \hspace{-0.3cm} f^{N, k-1}_{i, m ,V}(x) \equiv \min \left\{ f^{N, k-1}_{i, m}(z) \left| z \in S^{N,k}_{-, x} \cap V^{N,k-1}_- \right. \right\}, \\
& \hspace{-0.3cm} M^{N,k}_{D_i}(x) \! \equiv \! \max \left\{ D_i(z,y) \left| (z,y) \in P^{N,k}_{-, x} \right. \right\}, m^{N,k}_{D_i}(x) \! \equiv \! \min \left\{ D_i(z,y) \left| (z,y) \in P^{N,k}_{-, x} \right. \right\}
\end{split}
\end{equation}
$S^{N,k}_{-, x}$ and $P^{N,k}_{-, x}$ for $x \in V^{N,k}_-$ are given by
\begin{flalign}
\label{eq:A4}
& S^{N,k}_{-, x} \equiv \left\{ z \in S_- \left| 0 \leq z_m - x_{m} \leq {\alpha / 2^N}, \right. \right. \\ & \hspace{1 cm} \left. m_{C_l} ( z_m - x_{m}) \leq z_l - x_l \leq M_{C_l} ( z_m - x_{m}) \right\} \nonumber \\
& P^{N,k}_{-, x} \equiv \left\{ (z,y) \left| z \in S^{N,k}_{-, x} ,\ f^{N, k-1}_{i, m, V}(x) \! - \! M_{\Vert D \Vert} {\alpha \over 2^N} \leq y_i \leq f^{N, k-1}_{i, M,V}(x) + M_{\Vert D \Vert} {\alpha \over 2^N}, i = 1, ..., n \right. \right\} \nonumber &&
\end{flalign}
$f^{N, k-1}_{i, M, V_{\text{res.}}}(x)$ and $f^{N, k-1}_{i, m, V_{\text{res.}}}(x)$ for $x \in V^{N,k}_-$ given by
\begin{flalign}
\label{eq:A5}
& f^{N, k-1}_{i,M, V_{\text{res.}}}(x) \equiv \max \left\{ f^{N, k-1}_{i, M}(z) \left| z \in V^{N,k-1}_{\text{res.} , i, x} \right. \right\} \nonumber \\
& f^{N, k-1}_{i, m, V_{\text{res.}}}(x) \equiv \min \left\{ f^{N, k-1}_{i, m}(z) \left| z \in V^{N,k-1}_{\text{res.} , i, x} \right. \right\}  \\
& V^{N, k-1}_{\text{res.} ,i, x} \equiv \left\{ z \in S^{N,k}_{-, x} \cap V^{N,k-1}_- \left| m^{N,k}_{C_{il}}(x) \alpha / 2^N  \leq z_l - x_l \leq M^{N,k}_{C_{il}}(x) \alpha / 2^N \right. \right\} \nonumber
\end{flalign}
and $M^{N,k}_{C_{il}}(x)$ and $m^{N,k}_{C_{il}}(x)$ given by
\begin{flalign}
\label{eq:A6}
\hspace{-0.3cm} M^{N,k}_{C_{il}}(x) \equiv \max \left\{ C_{il}(z,y) \left| (z,y) \in P^{N,k}_{-, x} \right. \right\}, \ \ m^{N,k}_{C_{il}}(x) \equiv \min \left\{ C_{il}(z,y) \left| (z,y) \in P^{N,k}_{-, x} \right. \right\}
\end{flalign}
Relation \ref{eq:3.52} is modified to
\begin{flalign}
\label{eq:A7}
& f^{N,k}_i (x) = f^{N,k-1}_i ( x + C_i(x^{\nu} , f^{N,k-1}(x^{\nu})) \alpha/ 2^N ) - D_i(x^{\nu} , f^{N,k-1}(x^{\nu})) \alpha /2^N \nonumber \\
& f^{N,0}_i \equiv I_i , \ x \in V^{N,k}_- , \ x^{\nu}  = x + \nu {\alpha \over 2^N} , \ \nu = (m_{C_l} + M_{C_l}) \hat{e}_l /2 + \hat{e}_m
\end{flalign}
Relations \ref{eq:3.18} and \ref{eq:3.20} are equivalently valid with $\alpha > 0$ being the extent which, in general, the solution can be constructed below the initial condition hyperplane and $\Delta f^{N,k}$ a bound for $f^{N, k}_{i, M}(x) - f^{N, k}_{i, m}(x)$ for $x \in V^{N,k}_-$ and $L^{N,k}$ being the Lipschitz constant of $f^{N, k}_{i, M}$ or $f^{N, k}_{i, m}$ defined in \ref{eq:A2} on $V^{N,k}_-$.
\section{}
\label{appen:B}
In this Appendix we will show in detail that the bounds set for the solution in Section \ref{sec:3} at the $N +1$ step of the partitioning lie within the bounds of the $N$ step of the partitioning. $P^{N,k}_{+,x}$ for $x \in V^{N,k}$ in Section \ref{sec:3} was defined such that if $f$ is a solution to the system of PDE \ref{eq:3.1} subject to the initial condition and its characteristic curves $x^{(i)}(t)$ for $i = 1, ..., n$ pass through the point $x$, $x^{(i)}(x_m) = x$ with $x_m = x_{0m} + k \alpha /2^N$, then $(x^{(i)}(t), f(x^{(i)}(t))) \in P^{N,k}_{+,x}$ for $ - \alpha /2^N \leq t - x_{m} \leq 0$, therefore with defining $P^{N}_+ \equiv \cup^{2^N}_{k=1} (\cup_{z \in V^{N,k}} P^{N,k}_{+,z})$ we have $(z, f(z)) \in P^{N}_+$, $\forall z \in S_+$, i.e. the graph of the solution on $S_+$ is a subset of $P^{N}_+$. Here we will show that $P^{N+1}_+ \subseteq P^N_+$.

Lets assume (i): $f^{N, k-1}_{i, m}(\bar{x}) \leq f^{N+1, 2(k-1)}_{i, m}(\bar{x})$ and $f^{N+1, 2(k-1)}_{i, M}(\bar{x}) \leq f^{N, k-1}_{i, M}(\bar{x})$ for $\bar{x} \in V^{N,k-1} = V^{N+1,2(k-1)}$ (note that this is true for $k-1 = 0$) and try to prove (ii): $f^{N, k}_{i,m}(x) \leq f^{N+1, 2k}_{i,m}(x)$ and $f^{N+1, 2k}_{i, M}(x) \leq f^{N, k}_{i, M}(x)$ for $x \in V^{N,k} = V^{N+1,2k}$.

Here we show $f^{N+1, 2k}_{i, M}(x) \leq f^{N, k}_{i, M}(x)$, the proof of $f^{N, k}_{i,m}(x) \leq f^{N+1, 2k}_{i,m}(x)$ is similar. Based on the definitions in Section \ref{sec:3}
\vspace*{-0.3cm}
\begin{flalign}
\label{eq:B1}
& f^{N+1, 2k}_{i, M}(x) = f^{N+1, 2k-1}_{i, M}(x^i_1 ) + M^{N+1,2k}_{D_i}(x) {\alpha / 2^{N+1}} \nonumber \\ & \hspace{1.7cm} = f^{N+1, 2k-2}_{i, M}(x^i_2) + M^{N+1,2k-1}_{D_i}(x^i_1) {\alpha / 2^{N+1}} + M^{N+1,2k}_{D_i}(x) {\alpha / 2^{N+1}} \\ &  f^{N, k}_{i, M}(x) = f^{N, k-1}_{i, M}(x^i ) + M^{N,k}_{D_i}(x) {\alpha / 2^{N}} \nonumber
\end{flalign}
$x^i_1 \in V^{N+1, 2k-1}_{\text{res.} ,i, x}$, $x^i_2 \in V^{N+1, 2k-2}_{\text{res.} ,i, x^i_1 }$ and $x^i \in V^{N, k-1}_{\text{res.} ,i,x}$ are the points which $f^{N+1, 2k-1}_{i, M}$, $f^{N+1, 2k-2}_{i, M}$ and $f^{N,k-1}_{i, M}$ assume their maximum values in $V^{N+1, 2k-1}_{\text{res.} ,i,x}$, $V^{N+1, 2k-2}_{\text{res.} ,i,x^i_1 }$ and $V^{N, k-1}_{\text{res.} ,i,x}$ respectively. It can be shown that $V^{N+1, 2k-2}_{\text{res.} ,i,x^i_1} \subseteq V^{N, k-1}_{\text{res.} ,i,x}$ therefore from the assumption $f^{N+1, 2(k-1)}_{i,M}(\bar{x}) \leq f^{N, k-1}_{i, M}(\bar{x})$, $\bar{x} \in V^{N,k-1}$ it follows that $f^{N+1, 2k-2}_{i, M}(x^i_2) \leq f^{N, k-1}_{i, M}(x^i)$. Also it can be shown that $P^{N+1, 2k}_{+,x} \subseteq P^{N, k}_{+,x}$ and $P^{N+1, 2k-1}_{+, x^i_1} \subseteq P^{N, k}_{+,x}$ therefore $M^{N+1,2k}_{D_i}(x) \leq M^{N,k}_{D_i}(x)$ and $M^{N+1,2k-1}_{D_i}(x^i_1) \leq M^{N,k}_{D_i}(x)$, this proves $f^{N+1, 2k}_{i, M}(x) \leq f^{N, k}_{i, M}(x)$. \\

To complete the proof we have to show $V^{N+1, 2k-2}_{\text{res.} ,i , x^i_1 } \subseteq V^{N, k-1}_{\text{res.} ,i, x}$ , $P^{N+1, 2k}_{+,x} \subseteq P^{N, k}_{+, x}$ and $P^{N+1, 2k-1}_{+,x^i_1} \subseteq P^{N, k}_{+,x}$. Take $x' \in S^{N+1,2k}_{+,x} \cap V^{N+1,2k-1}$($= S^{N,k}_{+,x} \cap V^{N+1,2k-1}$) , it is clear that $S^{N+1,2k}_{+,x} \subset S^{N,k}_{+,x}$ and $S^{N+1,2k-1}_{+,x'} \subset S^{N,k}_{+,x}$ from their definitions given by \ref{eq:3.6}, also lets review the definitions of $P^{N+1,2k-1}_{+,x'}$, $P^{N,k}_{+,x}$ and $P^{N+1,2k}_{+,x}$
\vspace*{-0.2cm}
\begin{flalign}
\label{eq:B2}
& P^{N+1,2k}_{+,x} = \left\{ (z,y) \left| z \in S^{N+1,2k}_{+,x} ,  f^{N+1, 2k-1}_{i, m, V}(x) - M_{\Vert D \Vert} {\alpha \over 2^{N+1}}  \leq y_i \leq f^{N+1, 2k-1}_{i, M, V}(x) + M_{\Vert D \Vert} {\alpha \over 2^{N+1}} \right. \right\} \nonumber \\ & P^{N,k}_{+,x} = \left\{ (z,y) \left| z \in S^{N,k}_{+,x} ,  f^{N, k-1}_{i, m, V}(x) - M_{\Vert D \Vert} {\alpha \over 2^{N}}  \leq y_i \leq f^{N, k-1}_{i, M, V}(x) + M_{\Vert D \Vert} {\alpha \over 2^{N}} \right. \right\} \\ & P^{N+1,2k-1}_{+,x'} \! \! = \! \left\{ (z,y) \left| z \in S^{N+1,2k-1}_{+,x'} \! \! ,  f^{N+1, 2k-2}_{i, m, V}(x') - M_{\Vert D \Vert} {\alpha \over 2^{N+1}} \! \leq y_i \! \leq f^{N+1, 2k-2}_{i, M, V}(x') + \! M_{\Vert D \Vert} {\alpha \over 2^{N+1}} \right. \! \! \right\} \nonumber
\end{flalign}
from \ref{eq:B2} it is clear that $P^{N+1,2k-1}_{+,x'} \subset P^{N,k}_{+,x}$, since $\{ S^{N+1,2k-1}_{+,x'} \subset S^{N,k}_{+,x} \}$ and $ \{ f^{N+1, 2k-2}_{i, m, V}(x') \geq f^{N, k-1}_{i, m, V}(x)$ and $f^{N+1, 2k-2}_{i, M, V}(x') \leq f^{N, k-1}_{i, M, V}(x)$ by assumption (i) and the fact that $ (S^{N+1,2k-1}_{+,x'} \cap V^{N+1,2k-2} ) \subset (S^{N,k}_{+,x} \cap V^{N,k-1}) \} $. Now because $x^i_1 \in V^{N+1, 2k-1}_{\text{res.} ,i, x} \subseteq S^{N+1,2k}_{+,x} \cap V^{N+1,2k-1}$ this shows that $P^{N+1,2k-1}_{+,x^i_1} \subset P^{N,k}_{+,x}$. Also $P^{N+1,2k}_{+,x} \subset P^{N,k}_{+,x}$ since $S^{N+1,2k}_{+,x} \subset S^{N,k}_{+,x}$ and
\vspace*{-0.2cm}
\begin{equation}
\label{eq:B3}
\begin{split}
f^{N+1, 2k-1}_{i, m}(x') = f^{N+1, 2k-2}_{i, m, V_{\text{res.}}}(x') + m^{N +1,2k-1}_{D_i}(x') \alpha /2^{N+1} \geq f^{N, k-1}_{i, m, V}(x) - M_{\Vert D \Vert} \alpha /2^{N+1} \\ f^{N+1, 2k-1}_{i, M}(x') = f^{N+1, 2k-2}_{i, M, V_{\text{res.}}}(x') + M^{N +1,2k-1}_{D_i}(x') \alpha /2^{N+1} \leq f^{N, k-1}_{i, M, V}(x) + M_{\Vert D \Vert} \alpha /2^{N+1}
\end{split}
\end{equation}
the last inequalities in the above relation follow from the fact that $V^{N+1, 2k-2}_{\text{res.}, i, x'} \subset (S^{N+1,2k-1}_{+,x'} \cap V^{N+1,2k-2} ) \subset (S^{N,k}_{+,x} \cap V^{N,k-1}) $. From \ref{eq:B3} it can be concluded that $f^{N+1, 2k-1}_{i, m, V}(x) - M_{\Vert D \Vert} {\alpha / 2^{N+1}} \geq f^{N, k-1}_{i, m, V}(x) - M_{\Vert D \Vert} \alpha /2^{N}$ and $f^{N+1, 2k-1}_{i, M, V}(x) + M_{\Vert D \Vert} {\alpha / 2^{N+1}} \leq f^{N, k-1}_{i, M, V}(x) + M_{\Vert D \Vert} \alpha /2^{N}$ since \ref{eq:B3} holds for all $x' \in S^{N+1,2k}_{+,x} \cap V^{N+1,2k-1} $, therefore this shows $P^{N+1,2k}_{+,x} \subset P^{N,k}_{+,x}$.

To show $V^{N+1, 2k-2}_{\text{res.} ,i , x^i_1 } \subseteq V^{N, k-1}_{\text{res.} ,i, x}$ , lets review their definitions
\begin{flalign}
\label{eq:B4}
& V^{N+1, 2k-2}_{\text{res.} ,i , x^i_1 }\!\! \equiv \! \left\{ z \! \in \! S^{N+1,2k-1}_{+, x^i_1} \! \cap \! V^{N+1,2k-2} \left|  -  M^{N+1,2k-1}_{C_{il}}(x^i_1) {\alpha \over 2^{N+1}} \! \leq \! z_l \! - \! x^i_{1l} \leq \! - m^{N+1,2k-1}_{C_{il}}(x^i_1) {\alpha \over 2^{N+1}} \right. \right\} \nonumber \\ & V^{N, k-1}_{\text{res.} ,i, x} \equiv \left\{ z \in S^{N,k}_{+, x} \cap V^{N,k-1} \left| - M^{N,k}_{C_{il}}(x) \alpha / 2^N  \leq z_l - x_l \leq - m^{N,k}_{C_{il}}(x) \alpha / 2^N \right. \right\}
\end{flalign}
and $V^{N+1, 2k-1}_{\text{res.} ,i, x}$ is given by
\begin{flalign}
\label{eq:B5}
 V^{N+1, 2k-1}_{\text{res.} ,i, x} \!\! \equiv \! \left\{ z \! \in \! S^{N+1,2k}_{+, x} \cap V^{N+1,2k-1} \! \left|  -  M^{N+1,2k}_{C_{il}}(x) {\alpha \over 2^{N+1}} \! \leq z_l \! - \! x_l \leq \! - m^{N+1,2k}_{C_{il}}(x) {\alpha \over 2^{N+1}} \right. \! \right\}
\end{flalign}
for $x^i_1 \in V^{N+1, 2k-1}_{\text{res.} ,i, x}$ adding the inequalities in \ref{eq:B5} and in the first relation of \ref{eq:B4} we can conclude that if $z \in V^{N+1, 2k-2}_{\text{res.} ,i , x^i_1 }$ then
\begin{flalign}
\label{eq:B6}
- {1 \over 2} \! \left( M^{N+1,2k}_{C_{il}}(x) + M^{N+1,2k-1}_{C_{il}}(x^i_1) \right) {\alpha \over 2^{N}} \! \leq \! z_l - x_l \leq \! - {1 \over 2} \! \left( m^{N+1,2k-1}_{C_{il}}(x^i_1) + m^{N+1,2k}_{C_{il}}(x) \right) {\alpha \over 2^{N}}
\end{flalign}
as previously shown $P^{N+1,2k}_{+,x} \subset P^{N,k}_{+,x}$, $P^{N+1,2k-1}_{+,x^i_1} \subset P^{N,k}_{+,x}$ therefore
\begin{equation}
\label{eq:B7}
\begin{split}
& - {1 \over 2} \! \left( m^{N+1,2k-1}_{C_{il}}(x^i_1) + m^{N+1,2k}_{C_{il}}(x) \right) \leq - m^{N,k}_{C_{il}}(x) \\ & -{1 \over 2} \! \left( M^{N+1,2k-1}_{C_{il}}(x^i_1) + M^{N+1,2k}_{C_{il}}(x) \right) \geq - M^{N,k}_{C_{il}}(x) 
\end{split}
\end{equation}
this shows $V^{N+1, 2k-2}_{\text{res.} ,i , x^i_1 } \subseteq V^{N, k-1}_{\text{res.} ,i, x}$ . From the discussion above it is clear that $P^{N+1}_+ \subseteq P^N_+$.


\end{document}